\documentclass[11pt]{article}

\usepackage{amsmath}
\usepackage{amsthm}
\usepackage{amsfonts}
\usepackage{amssymb}
\usepackage{enumitem}
\usepackage{array}
\usepackage{setspace}
\usepackage[labelformat=simple]{subcaption}

\usepackage{color}
\usepackage{relsize}
\usepackage{hyperref}

\usepackage{url}
\usepackage{graphicx}
\usepackage[numbers,sort&compress]{natbib}
\usepackage[margin=1.5 in]{geometry}
\newtheorem{theorem}{Theorem}[section]
\newtheorem{proposition}[theorem]{Proposition}
\newtheorem{lemma}[theorem]{Lemma}
\newtheorem{corollary}[theorem]{Corollary}

\newtheorem{fact}[theorem]{Fact}

\theoremstyle{definition}
\newtheorem{definition}[theorem]{Definition}

\newtheorem{question}[theorem]{Question}

\theoremstyle{remark}
\newtheorem{remark}[theorem]{Remark}

\newcommand{\ceeil}[1]{\left \lceil #1 \right \rceil }
\newcommand{\flooor}[1]{\left \lfloor #1 \right \rfloor }

\newcommand{\bigboxtimes}{\mathop{\mathlarger{\mathlarger{\boxtimes}}}
\displaylimits}
\newcommand{\super}{\mathcal{U}}
\DeclareMathOperator{\ter}{ter}

\usepackage{tikz}
\usetikzlibrary{calc}
\usetikzlibrary{decorations.pathreplacing}
\usetikzlibrary{decorations.markings}

\tikzstyle{vertex}=[circle, draw, inner sep=0pt, minimum size=4pt,fill=black]
\newcommand{\vertex}{\node[vertex]}
\tikzstyle{whitevertex}=[circle, draw, inner sep=0pt, minimum size=4pt,fill=white]
\newcommand{\whitevertex}{\node[whitevertex]}
\tikzstyle{hollowvertex}=[circle, draw, inner sep=0pt, minimum size=4pt, fill=white]

\tikzstyle{phantomvertex}=[circle, draw, inner sep=0pt, minimum size=4pt,color=white]

\usetikzlibrary{backgrounds}
\usetikzlibrary{arrows.meta}
\tikzset{
  .../.tip={[sep=2pt 2]
    Round Cap[]. Circle[length=0pt 2,sep=2pt] Circle[length=0pt 2,sep=2pt] Circle[length=0pt 2, sep=2pt 2]}}
\tikzset{
  .../.tip={[sep=2pt 2]
    Round Cap[]. Circle[length=0pt 2,sep=2pt] Circle[length=0pt 2,sep=2pt] Circle[length=0pt 2, sep=2pt 2]}}
\tikzset{
      ellipsis/.tip={
Square[length=2pt,sep=0pt,color=white] Circle[length=1pt,sep=0pt,color=black] Square[length=1pt,sep=0pt,color=white]
Circle[length=1pt,sep=0pt,color=black] Square[length=1pt,sep=0pt,color=white]
Circle[length=1pt,sep=0pt,color=black] Square[length=2pt,sep=0pt,color=white]}}
\tikzset{middlearrow/.style n args={2}{
        decoration={markings,
            mark= at position {#2} with {\arrow{#1}} ,
        },
        postaction={decorate}
    }
}

\begin{document}
	\title{The Threshold Dimension of a Graph}
\author{Lucas Mol\footnote{University of Winnipeg, 515 Portage Avenue, Winnipeg, MB R3B 2E9}, Matthew J. H. Murphy\footnote{University of Toronto, 27 King's College Circle, Toronto, ON M5S 1A1}, and Ortrud R.\ Oellermann$^*$\thanks{Supported by an NSERC Grant CANADA, Grant number RGPIN-2016-05237}\\
	%		\\ University of Winnipeg\\
	%		\small 515 Portage Avenue, Winnipeg, MB R3B 2E9\\
	\small \href{mailto:l.mol@uwinnipeg.ca}{l.mol@uwinnipeg.ca}, \href{mailto:mattjames.murphy@mail.utoronto.ca}{mattjames.murphy@mail.utoronto.ca},\\ \small\href{mailto:o.oellermann@uwinnipeg.ca}{o.oellermann@uwinnipeg.ca}}
\date{}

\maketitle
\begin{abstract}
	Let $G$ be a graph, and let $u$, $v$, and $w$ be vertices of $G$. If the distance between $u$ and $w$ does not equal the distance between $v$ and $w$, then $w$ is said to \emph{resolve} $u$ and $v$.	The \emph{metric dimension} of $G$, denoted $\beta(G)$, is the cardinality of a smallest set $W$ of vertices such that every pair of vertices of $G$ is resolved by some vertex of $W$.
The \emph{threshold dimension} of a graph $G$, denoted $\tau(G)$,  is the minimum metric dimension among all graphs $H$ having $G$ as a spanning subgraph. In other words, the threshold dimension of $G$ is the minimum metric dimension among all graphs obtained from $G$ by adding edges.  If $\beta(G) = \tau(G)$, then $G$ is said to be \emph{irreducible}; otherwise, we say that $G$ is \emph{reducible}. If $H$ is a graph having $G$ as a spanning subgraph and such that $\beta(H)=\tau(G)$, then $H$ is called a \emph{threshold graph} of $G$.
	
The first main part of the paper  has a geometric flavour, and gives an expression for the threshold dimension of a graph in terms of a minimum number of strong products of paths (each of sufficiently large order) that admits a certain type of embedding of the graph. This result is used to show that there are trees of arbitrarily large metric dimension having threshold dimension $2$. The second main part of the paper focuses on the threshold dimension of trees. A sharp  upper bound for the threshold dimension of trees is established. It is also shown that  the irreducible trees are precisely those of metric dimension at most 2. Moreover, if $T$ is a tree with metric dimension 3 or 4, then $T$ has threshold dimension $2$. It is shown, in these two cases, that a threshold graph for $T$ can be obtained by adding exactly one or two edges to $T$, respectively. However, these results do not extend to trees with metric dimension $5$, i.e., there are trees of metric dimension $5$ with threshold dimension exceeding $2$.
%We first give a more geometric description of the threshold dimension of a graph in terms of certain embeddings in strong products of paths.  This allows us to show that there are trees of arbitrarily large metric dimension having threshold dimension $2$.  We establish a sharp upper bound for the threshold dimension of a tree of order $n$.  We show that every tree $T$ with  $\beta(T) \geq 3$ is reducible; in fact, for every tree $T$ with $\beta(T) \ge 3$, we show that there is an edge $e$ whose addition to $T$ decreases the metric dimension.  It follows that all trees with metric dimension $3$ have threshold dimension $2$.  Finally, we show that all trees with metric dimension $4$ have threshold dimension $2$, and that there are trees with metric dimension $5$ and threshold dimension greater than $2$.
\end{abstract}

\section{Introduction}
Slater \cite{Slater1975}, being motivated by the problem of uniquely determining the location of an intruder in a network, first introduced the notion of `resolvability' in graphs.  For vertices $x$ and $y$ of a graph $G$, let $d_G(x,y)$ denote the distance between $x$ and $y$ in $G$.  We write $d(x,y)$ in place of $d_G(x,y)$ if $G$ is clear from context.  A vertex $w$ is said to \emph{resolve} a pair $u,v$ of vertices in $G$ if $d(u,w)\neq d(v,w)$.
A set $W \subseteq V(G)$ of vertices \emph{resolves} the graph $G$, and we say that $W$ is a \emph{resolving set} for $G$, if every pair of vertices of $G$ is resolved by some vertex of $W$. A smallest  resolving set of $G$ is called a \emph{basis} of $G$, and its cardinality is called the \emph{metric dimension} of $G$, denoted $\beta(G)$. Since being introduced by Slater \cite{Slater1975}, and independently by Harary and Melter \cite{HararyMelter1976},  the metric dimension has been studied extensively.   It is well-known that the problem of determining the metric dimension of a graph is NP-hard; see, for example, the proof of Khuller et al.~\cite{Khulleretal1996}. This suggests studying the problem of finding the metric dimension for special classes of graphs, and developing heuristics for approximating this invariant. A formula for the metric dimension of trees has been (re)discovered several times~\cite{Chartrandetal2000, HararyMelter1976, Slater1975}. A variety of applications and a substantial collection of publications that emphasize both the theoretical and computational aspects of this invariant are highlighted, for example, in the works of C\'{a}ceres et al.~\cite{Caceresetal2007} and Belmonte et al.~\cite{Belmonteetal2015}.
When we say dimension in this paper, unless qualified, we are referring to the metric dimension.

The question of how the metric dimension of a graph relates to that of its subgraphs has been considered by Chartrand et al.~\cite{Chartrandetal2000} and Khuller et al.~\cite{Khulleretal1996}. In particular, Chartrand et al.~\cite{Chartrandetal2000} proved that for every $\epsilon >0$, there is a graph $H$, and a subgraph $G$ of $H$, such that $\beta(H)/\beta(G) < \epsilon$.  Khuller et al.~\cite{Khulleretal1996} established a lower bound for the metric dimension of a graph in terms of its clique number.  It is natural to ask by how much, if at all, we can reduce the metric dimension of a graph by adding edges.  In other words, for a given graph $G$, what is the smallest metric dimension among all graphs that contain $G$ as a spanning subgraph?  We let $\super(G)$ denote the set of all graphs $H$ that contain $G$ as a spanning subgraph.

\begin{definition}
The \emph{threshold dimension} of a graph $G$, denoted $\tau(G)$, is the minimum metric dimension among all graphs $H$ that contain $G$ as a spanning subgraph, i.e., we have $\tau(G)=\min\{\beta(H)\colon\ H\in \super(G)\}$.  A graph $H\in\super(G)$ of metric dimension $\tau(G)$ is called a \emph{threshold graph} of $G$.
\end{definition}

For a graph $G$, let $\overline{G}$ denote the complement of $G$.  The graph obtained from $G$ by adding a set $E \subseteq E\left(\overline{G}\right)$ of edges to $G$ is denoted by $G+E$.  For a single edge $e\in E\left(\overline{G}\right)$, we use the notation $G+e$ in place of $G+\{e\}$.  Evidently, we may write
\[
\tau(G)=\min\{\beta(G+E)\colon\ E\subseteq E\left(\overline{G})\right)\}.
\]

The threshold dimension of $G$ (and the threshold dimension of any spanning subgraph of $G$) gives a lower bound for the metric dimension of $G$. This suggests the problem of determining those graphs for which this lower bound is achieved.

\begin{definition}
A graph $G$ is said to be \emph{irreducible} if $\beta(G) =\tau(G)$;  otherwise, the graph $G$ is said to be \emph{reducible}.
\end{definition}

Before we proceed, we give some examples.  It is well-known that for every $n\geq 2$, and every connected graph $G$ of order $n$, we have
\[
1\leq \beta(G)\leq n-1,
\]
with equality on the left if and only if $G\cong P_n$, and equality on the right if and only if $G\cong K_n$.
For every $n\geq 2$, the path $P_n$, the unique connected graph of order $n$ and dimension $1$, is irreducible.  Further, since the path $P_n$ does not contain any other connected graph of order $n$ as a spanning subgraph, we see that all connected graphs of order $n$ and dimension $2$ are also irreducible. At the other extreme, the complete graph $K_n$, the unique connected graph of order $n$ and dimension $n-1$, is trivially irreducible since $\super(K_n)=\{K_n\}$.  However, graphs of order $n$ and dimension $n-2$ need not be irreducible.  For example, for $n\geq 5$ and $s\in\{1,\ldots,n-1\}$, the complete bipartite graph $K_{s,n-s}$ has dimension $n-2$, but the addition of a single edge produces a graph of dimension $n-3$.

Chartrand et al.~\cite{Chartrandetal2000} proved that if $T$ is a tree of order at least $3$, then for every edge $e \in E\left(\overline{T}\right)$, we have $\beta(T) -2 \leq \beta(T+e) \leq \beta(T)+1$.  For a graph $G$, we will be interested in whether there exists a single edge $e\in E\left(\overline{G}\right)$ such that $\beta(G+e)<\beta(G)$.  If such an edge does exist, then $G$ is obviously reducible, but we will see that this is not a necessary condition for reducibility.

In Section \ref{embed}, we provide a geometric interpretation of the threshold dimension of a graph in terms of a minimum number of strong products of paths (each of sufficiently large order) that admit a certain type of embedding of the graph. We apply this result to demonstrate that there are trees of arbitrarily large dimension with threshold dimension $2$.  Finally, we compare the threshold dimension to the \emph{strong isometric dimension}, see \cite{FitzpatrickNowakowski2000}, which is also defined in terms of embeddings in strong products of paths.

In Section \ref{trees}, we focus on the threshold dimension of trees.
We first determine a sharp upper bound for the threshold dimension of every tree of order $n$.  We then show that if $T$ is a tree with $\beta(T) \ge 3$, then there is an edge $e\in E\left(\overline{T}\right)$ such that $\beta(T+e) < \beta(T)$.
 %We also show that this result does not extend to reducible graphs in general by showing that there is a reducible graph $G$ such that for every $e \in E\left(\overline{G}\right)$ we have $\beta(G+e) \geq \beta(G)$.
For every tree $T$ with dimension $4$, we show that there is a set of two edges whose addition to $T$ produces a graph with dimension $2$.  Thus, if a tree $T$ has $\beta(T) \in \{2,3,4\}$, then $\tau(T)=2$.  Finally, we show that there are trees with dimension $5$ having threshold dimension strictly greater than $2$.

\section{Preliminaries}

For a graph $G$, let $\mbox{diam}(G)$ denote the \emph{diameter} of $G$, i.e., the maximum distance between a pair of vertices of $G$.
The $k$-\emph{neighbourhood} of a vertex $v$ in $G$, denoted $N_k(v)$, is the set of vertices in $G$ which are distance exactly $k$ from $v$,  i.e., we have $N_k(v)=\{x \in V(G)\colon\ d(x,v)=k\}$.  The notation $N(v)$ is used in place of $N_1(v)$.  For a set $W$ of vertices of $G$, the $W$-\emph{neighbourhood} of a vertex $v$ in $G$ is defined as $N_W(v)=N(v) \cap W$.

If $G_1,G_2,\dots, G_k$ are graphs, then their \emph{strong product} is the graph
\[
G_1\boxtimes G_2\boxtimes\cdots\boxtimes G_k=\bigboxtimes_{i=1}^k G_i,
\]
with vertex set $\{(x_1,x_2,\dots,x_k)\colon\ x_i\in V(G_i)\}$, and for which two distinct vertices $x=(x_1,x_2,\dots,x_k)$ and $y=(y_1,y_2,\dots, y_k)$ are adjacent if and only if for every $1\leq i\leq k$, either $x_iy_i\in E(G_i)$ or $x_i=y_i$.  The distance between $x$ and $y$ in $G_1\boxtimes G_2\boxtimes\cdots\boxtimes G_k$ is given by $\max\{d_{G_i}(x_i,y_i)\colon\ 1\leq i\leq k\}.$  For a graph $G$, we let $G^{\boxtimes, k}$ denote the $k$th power of $G$ with respect to the strong product, i.e., we have
\[
G^{\boxtimes,k}=\displaystyle\bigboxtimes_{i=1}^k G.
\]
See~\cite{ProductHandbook} for more background on graph products.

Let $G$ and $H$ be graphs.  A map $\varphi:V(G)\rightarrow V(H)$ is called an \emph{embedding} of $G$ in $H$ if it is injective and preserves the edge relation (i.e., for all vertices $x,y\in V(G)$, if $xy\in E(G)$, then $\varphi(x)\varphi(y)\in E(H)$). The map $\varphi$ is an \emph{isometric embedding} of $G$ in $H$ if $d_G(u,v)=d_H(\varphi(u),\varphi(v))$ for all $u,v\subseteq V(G)$.  Note that an isometric embedding of $G$ in $H$ is necessarily an embedding of $G$ in $H$.  If $G$ is a subgraph of $H$, then we say that $G$ is \emph{isometric} in $H$ if $d_G(u,v)=d_H(u,v)$ for all vertices $u,v\in V(G)$, i.e., if the inclusion map from $V(G)$ to $V(H)$ is an isometric embedding of $G$ in $H$.

For a graph $G$ and a subset $W\subseteq V(G)$, we let $G[W]$ denote the subgraph of $G$ induced by $W$.  For an embedding $\varphi$ of $G$ in $H$, we let $\varphi(G)=H[\varphi(V(G))]$, i.e., $\varphi(G)$ is the subgraph of $H$ induced by the range of $\varphi$.  Clearly, the graph $\varphi(G)$ is isomorphic to the graph $G'\in \super(G)$ with vertex set $V(G')=V(G)$ and edge set $E(G')=\{xy\colon\ \varphi(x)\varphi(y)\in E(\varphi(G))\}.$

The metric dimension of trees is well understood, and a metric basis for a tree can be constructed in polynomial time.  We require some terminology and notation to describe this procedure.  Let $T$ be a tree. A vertex $v$ of degree at least $3$ in $T$ is called a  \emph{major vertex}. A leaf $u$ is said to be a \emph{terminal vertex} of the major vertex $v$ if $d(u,v)< d(u,w)$ for all other major vertices $w$ of $T$. If $u$ is a terminal vertex of $v$, then the maximal path of $T$ containing $u$ but not $v$ is called a \emph{limb} at $v$.  The \emph{terminal degree} of $v$, denoted $\ter(v)$,  is the number of terminal vertices of $v$.
%A major vertex $v$ is called an \emph{exterior major vertex} if its terminal degree is non-zero.  Let $\sigma(T)$ be the sum of the terminal degrees of the major vertices of $T$, and let ex$(T)$ denote the number of exterior major vertices of $T$.
The following was proven by Slater~\cite{Slater1975}, and  independently by Harary and Melter~\cite{HararyMelter1976}.  A different proof was provided by Chartrand et al.~\cite{Chartrandetal2000}.

\begin{theorem}\label{Slater}
Let $T$ be a tree that is not isomorphic to a path, and let $S$ be the set of exterior major vertices of $T$.  Then $\beta(T)=\sum_{v \in S} (\ter(v)-1)$. Moreover, a basis for $T$ can be constructed by selecting, for each exterior major vertex $v$ with terminal degree at least $2$, exactly one vertex from all but one of its limbs.
\end{theorem}

\noindent
It follows that if $T$ is a tree with  $\beta(T) \geq \ell$, then $T$ must have more than $\ell$ leaves.

%We say that $G$ can be \emph{embedded} in $H$ if $G$ is isomorphic to some subgraph of $H$, i.e., if there exists an injective function $\varphi: V(G) \rightarrow V(H)$ such that for all $x,y \in V(G)$, if $xy \in E(G)$ then $\varphi(x)\varphi(y) \in E(H)$. We then say that $\varphi$ is an embedding of $G$ in $H$, and denote the subgraph induced by the range of $\varphi$ as $\varphi(G)$. Clearly $G$ is a spanning subgraph of $\varphi(G)$.

\section{A geometric interpretation of the threshold dimension} \label{embed}

In this section, we present a geometric interpretation of the threshold dimension of a graph, in terms of certain embeddings in strong products of paths.  Throughout this section, we let $V(P_n)=\{0,\ldots,n-1\}$. Thus, the vertices of $P_n^{\boxtimes, k}$ are $k$-tuples over the set $\{0, \ldots, n-1\}$.  Our choice of notation for the vertex set of $P_n$ makes calculating distances in $P_n^{\boxtimes,k}$ particularly simple.

\begin{fact} \label{distanceinstrongproducts}
If $x=(x_1, \ldots, x_k)$ and $y=(y_1, \ldots, y_k)$ are in $V\left(P_n^{\boxtimes, k}\right)$, then
\[
d(x,y)=\mathrm{max}\{|x_i-y_i|\colon\ 1 \leq i \leq k\}.
\]
In particular, if $x$ and $y$ are distinct, then they are adjacent if and only if $|x_i-y_i|\leq 1$ for every $1 \leq i \leq k$.
\end{fact}

Our choice of notation for $V(P_n)$ is also important because the vertices of $P_n$ will correspond to distances, and hence vertices of $P_n^{\boxtimes,k}$ will correspond to vectors of distances.  Let $G$ be a connected graph with resolving set $W=\{w_1,w_2,\dots,w_k\}.$  Then every vertex $x\in V(G)$ is uniquely determined by its vector of distances to vertices in $W$, given explicitly by $\left(d_G(w_1,x),d_G(w_2,x),\dots,d_G(w_b,x)\right)$.  We first show that the map sending every vertex $x$ to this vector of distances to $W$ is an embedding of $G$ in $P^{\boxtimes,k}$ for some sufficiently large path $P$.

We will then show that if $W$ is a resolving set for some graph $H\in \super(G)$, then there is an embedding $\varphi$ of $G$ in $P^{\boxtimes,k}$ for some sufficiently large path $P$, such that for every vertex $x\in V(G)$, we have that $\varphi(x)$ is exactly the vector of distances in $\varphi(G)$ from $\varphi(x)$ to the vertices of $\varphi(W)$.  We first give a formal definition of the embeddings that we have just described.

\begin{definition}
Let $G$ be a graph, let $W=\{w_1,w_2,\dots,w_k\}$ be a subset of $V(G)$, and let $P$ be a path.  A \emph{$W$-resolved embedding} of $G$ in $P^{\boxtimes,k}$ is an embedding $\varphi$ of $G$ in $P^{\boxtimes,k}$ such that for every $x\in V(G)$, we have
\[
\varphi(x)=\left(d_{\varphi(G)}(\varphi(w_1),\varphi(x)),%d_{\varphi(G)}(\varphi(w_2),\varphi(x)),
\dots,d_{\varphi(G)}(\varphi(w_k),\varphi(x))\right),
\]
i.e., for every $1\leq i\leq k$, the $i$th coordinate of $\varphi(x)$ is exactly the distance between $\varphi(w_i)$ and $\varphi(x)$ in $\varphi(G)$.
\end{definition}

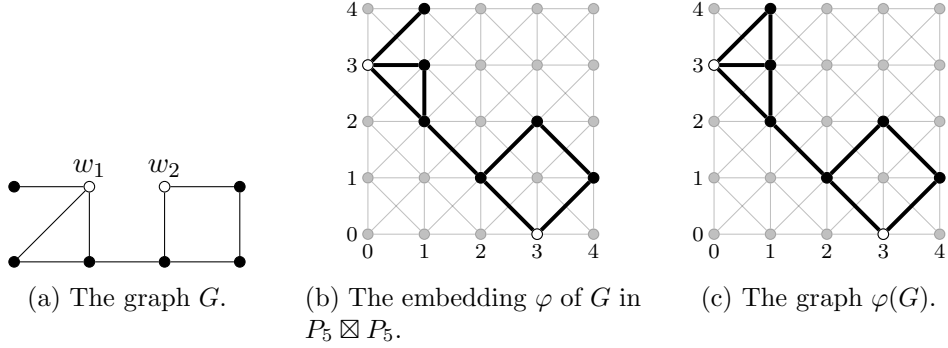
\begin{figure}[htb]
\centering
\begin{subfigure}[t]{0.32\textwidth}
\centering
\begin{tikzpicture}[scale=1]
\vertex (0) at (0,0) {};
%\node[below] at (0,0) {$[2,3]$};
\vertex (1) at (1,0) {};
%\node[below] at (1,0) {$[1,2]$};
\vertex (2) at (2,0) {};
%\node[below] at (2,0) {$[2,1]$};
\vertex (3) at (3,0) {};
%\node[below] at (3,0) {$[3,2]$};
\whitevertex (4) at (1,1) {};
\node[above] at (1,1) {$w_1$};
\whitevertex (5) at (2,1) {};
\node[above] at (2,1) {$w_2$};
\vertex (6) at (0,1) {};
\vertex (7) at (3,1) {};
\path
(0) edge (1)
(1) edge (2)
(2) edge (3)
(1) edge (4)
(2) edge (5)
(4) edge (6)
(4) edge (0)
(5) edge (7)
(3) edge (7);
\end{tikzpicture}
\caption{The graph $G$.}
\end{subfigure}
\begin{subfigure}[t]{0.32\textwidth}
\centering
\begin{tikzpicture}[scale=0.75]
  \foreach \x in {0,...,4}
    \foreach \y in {0,...,4}
       {
       \vertex[opacity=0.25]  (\x\y) at (\x,\y) {};
       }
  \foreach \x in {0,...,3}
    \foreach \y in {0,...,3}
    {
    \pgfmathtruncatemacro{\a}{\x+1}
    \pgfmathtruncatemacro{\b}{\y+1}
    \path[opacity=0.25]
    (\x\y) edge (\a\b)
    (\x\b) edge (\a\y);
    }
\foreach \x in {0,...,3}
    \foreach \y in {0,...,4}
    {
    \pgfmathtruncatemacro{\a}{\x+1}
    \path[opacity=0.25]
    (\x\y) edge (\a\y)
    (\y\x) edge (\y\a);
    }
\foreach \x in {0,...,4}
{
\node[left] at (0,\x) {\scriptsize \x};
\node[below] at (\x,0) {\scriptsize \x};
}
\whitevertex (a) at (0,3) {};
\vertex (b) at (1,2) {};
\vertex (c) at (1,3) {};
\whitevertex (d) at (3,0) {};
\vertex (e) at (2,1) {};
\vertex (f) at (3,2) {};
\vertex (g) at (1,4) {};
\vertex (h) at (4,1) {};
\path[line width=1.5pt]
(a) edge (b)
(b) edge (c)
(d) edge (e)
(e) edge (f)
(b) edge (e)
(a) edge (g)
(a) edge (c)
(d) edge (h)
(f) edge (h);
\end{tikzpicture}
\caption{The embedding $\varphi$ of $G$ in $P_5\boxtimes P_5$.}
\end{subfigure}
\begin{subfigure}[t]{0.32\textwidth}
\centering
\begin{tikzpicture}[scale=0.75]
  \foreach \x in {0,...,4}
    \foreach \y in {0,...,4}
       {
       \vertex[opacity=0.25]  (\x\y) at (\x,\y) {};
       }
  \foreach \x in {0,...,3}
    \foreach \y in {0,...,3}
    {
    \pgfmathtruncatemacro{\a}{\x+1}
    \pgfmathtruncatemacro{\b}{\y+1}
    \path[opacity=0.25]
    (\x\y) edge (\a\b)
    (\x\b) edge (\a\y);
    }
\foreach \x in {0,...,3}
    \foreach \y in {0,...,4}
    {
    \pgfmathtruncatemacro{\a}{\x+1}
    \path[opacity=0.25]
    (\x\y) edge (\a\y)
    (\y\x) edge (\y\a);
    }
\foreach \x in {0,...,4}
{
\node[left] at (0,\x) {\scriptsize \x};
\node[below] at (\x,0) {\scriptsize \x};
}
\whitevertex (a) at (0,3) {};
\vertex (b) at (1,2) {};
\vertex (c) at (1,3) {};
\whitevertex (d) at (3,0) {};
\vertex (e) at (2,1) {};
\vertex (f) at (3,2) {};
\vertex (g) at (1,4) {};
\vertex (h) at (4,1) {};
\path[line width=1.5pt]
(a) edge (b)
(b) edge (c)
(d) edge (e)
(e) edge (f)
(b) edge (e)
(a) edge (g)
(a) edge (c)
(d) edge (h)
(f) edge (h)
(c) edge (g);
\end{tikzpicture}
\caption{The graph $\varphi(G)$.}
%, defined by $\varphi(x)=(d_G(w_1,x),d_G(w_2,x))$ for all $x\in V(G)$.}
\end{subfigure}
\caption{A $\{w_1,w_2\}$-resolved embedding $\varphi$ of $G$ in $P_5\boxtimes P_5$, defined by $\varphi(x)=(d_G(w_1,x),d_G(w_2,x))$ for all $x\in V(G)$.}
\label{ResolvedEmbedding1}
\end{figure}

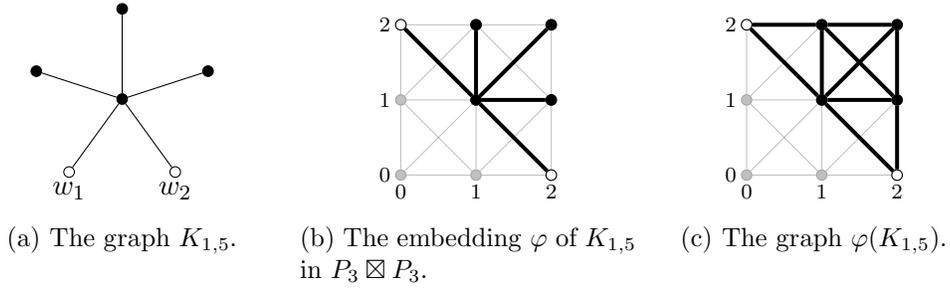
\begin{figure}[htb]
\centering
\begin{subfigure}[t]{0.32\textwidth}
\centering
\begin{tikzpicture}[scale=1.2]
\vertex (0) at (0,0) {};
\vertex (1) at (0+18:1) {};
\vertex (2) at (72+18:1) {};
\vertex (3) at (144+18:1) {};
\whitevertex (4) at (216+18:1) {};
\node[below] at (216+18:1) {$w_1$};
\whitevertex (5) at (288+18:1) {};
\node[below] at (288+18:1) {$w_2$};
\path
(0) edge (1)
(0) edge (2)
(0) edge (3)
(0) edge (4)
(0) edge (5);
\end{tikzpicture}
\caption{The graph $K_{1,5}$.}
\end{subfigure}
\begin{subfigure}[t]{0.32\textwidth}
\centering
\begin{tikzpicture}[scale=1]
  \foreach \x in {0,...,2}
    \foreach \y in {0,...,2}
       {
       \vertex[opacity=0.25]  (\x\y) at (\x,\y) {};
       }
  \foreach \x in {0,...,1}
    \foreach \y in {0,...,1}
    {
    \pgfmathtruncatemacro{\a}{\x+1}
    \pgfmathtruncatemacro{\b}{\y+1}
    \path[opacity=0.25]
    (\x\y) edge (\a\b)
    (\x\b) edge (\a\y);
    }
\foreach \x in {0,...,1}
    \foreach \y in {0,...,2}
    {
    \pgfmathtruncatemacro{\a}{\x+1}
    \path[opacity=0.25]
    (\x\y) edge (\a\y)
    (\y\x) edge (\y\a);
    }
\foreach \x in {0,...,2}
{
\node[left] at (0,\x) {\scriptsize \x};
\node[below] at (\x,0) {\scriptsize \x};
}
%\node[above] at (0,2) {$\varphi(w_1)$};
%\node[right] at (2,0) {$\varphi(w_2)$};
\whitevertex (a) at (0,2) {};
\vertex (b) at (1,1) {};
\whitevertex (c) at (2,0) {};
\vertex (d) at (1,2) {};
\vertex (e) at (2,2) {};
\vertex (f) at (2,1) {};
\path[line width=1.5pt]
(b) edge (a)
(b) edge (c)
(b) edge (d)
(b) edge (e)
(b) edge (f);
\end{tikzpicture}
\caption{The embedding $\varphi$ of $K_{1,5}$ in $P_3\boxtimes P_3$.}
\end{subfigure}
\begin{subfigure}[t]{0.32\textwidth}
\centering
\begin{tikzpicture}[scale=1]
  \foreach \x in {0,...,2}
    \foreach \y in {0,...,2}
       {
       \vertex[opacity=0.25]  (\x\y) at (\x,\y) {};
       }
  \foreach \x in {0,...,1}
    \foreach \y in {0,...,1}
    {
    \pgfmathtruncatemacro{\a}{\x+1}
    \pgfmathtruncatemacro{\b}{\y+1}
    \path[opacity=0.25]
    (\x\y) edge (\a\b)
    (\x\b) edge (\a\y);
    }
\foreach \x in {0,...,1}
    \foreach \y in {0,...,2}
    {
    \pgfmathtruncatemacro{\a}{\x+1}
    \path[opacity=0.25]
    (\x\y) edge (\a\y)
    (\y\x) edge (\y\a);
    }
\foreach \x in {0,...,2}
{
\node[left] at (0,\x) {\scriptsize \x};
\node[below] at (\x,0) {\scriptsize \x};
}
\whitevertex (a) at (0,2) {};
\vertex (b) at (1,1) {};
\whitevertex (c) at (2,0) {};
\vertex (d) at (1,2) {};
\vertex (e) at (2,2) {};
\vertex (f) at (2,1) {};
%\node[above] at (0,2) {$\varphi(w_1)$};
%\node[right] at (2,0) {$\varphi(w_2)$};
\path[line width=1.5pt]
(b) edge (a)
(b) edge (c)
(b) edge (d)
(b) edge (e)
(b) edge (f)
(a) edge (d)
(d) edge (e)
(d) edge (f)
(e) edge (f)
(f) edge (c);
\end{tikzpicture}
\caption{The graph $\varphi(K_{1,5})$.}
\end{subfigure}
\caption{A $\{w_1,w_2\}$-resolved embedding $\varphi$ of $K_{1,5}$ in $P_3\boxtimes P_3$.}
\label{ResolvedEmbedding2}
\end{figure}

Given a graph $G$ with a resolving set $W$, we will see that one can define a $W$-resolved embedding of $G$ by mapping each vertex $x\in V(G)$ to its vector of distances to vertices in $W$ (see Figure~\ref{ResolvedEmbedding1}).  However, this is not the only way that a $W$-resolved embedding of a graph $G$ may arise.  Figure~\ref{ResolvedEmbedding2} shows a $W$-resolved embedding of a graph $G$ in which $W$ is not a resolving set for $G$.  Note, however, that $\varphi(W)$ is a resolving set for the graph $\varphi(G)$ in this case.  In particular, this means that $W$ is a resolving set for the graph $G'\in\super(G)$ corresponding to $\varphi(G)$.  We will see that this observation is true in general.  The main result of this section is the following.

\begin{theorem}\label{Correspondence}
Let $G$ be a connected graph of diameter $D$, and let $W=\{w_1,w_2,\dots,$ $w_k\}\subseteq V(G)$.  Then $W$ is a resolving set for some graph $H\in \super(G)$ if and only if there is a $W$-resolved embedding of $G$ in $P_{D+1}^{\boxtimes,k}$.
\end{theorem}

The following corollary is immediate, and gives a  geometric interpretation of the threshold dimension.

\begin{corollary}\label{ThresholdEmbedding}
Let $G$ be a connected graph of diameter $D$. Then $\tau(G)$ is the minimum cardinality of a set $W\subseteq V(G)$ such that there is a $W$-resolved embedding of $G$ in $P^{\boxtimes,|W|}_{D+1}$.
\end{corollary}

Before we proceed with the proof of Theorem~\ref{Correspondence}, we prove that, given a connected graph $G$ of diameter $D$ with resolving set $W$, the map which sends every vertex to its vector of distances to $W$ is in fact a $W$-resolved embedding of $G$ in $P^{\boxtimes,k}_{D+1}$.

\begin{lemma} \label{charlemma1}
Let $G$ be a connected graph of diameter $D$ with resolving set $W=\{w_1,w_2,\dots,w_k\}\subseteq V(G)$.  Then there is a $W$-resolved embedding of $G$ in $P_{D+1}^{\boxtimes,k}$.
\end{lemma}

\begin{proof}
Define $\varphi:V(G)\rightarrow V\left(P_{D+1}^{\boxtimes,k}\right)$ by
\[
\varphi(x)=\left(d_G(w_1,x),d_G(w_2,x),\dots,d_G(w_k,x)\right)
\]
for all $x\in V(G)$.  Since $W$ resolves $G$, it follows immediately that $\varphi$ is injective.  Further, for every pair of distinct vertices $x$ and $y$ in $P_{D+1}^{\boxtimes,k}$, if $\varphi(x)\varphi(y)\not\in E\left(P_{D+1}^{\boxtimes,k}\right)$, then we must have $|d_G(w_i,x)-d_G(w_i,y)|>1$ for some $i\in\{1,\dots,k\}$, and thus $xy\not\in E(G)$.  So $\varphi$ is an embedding of $G$ in $P_{D+1}^{\boxtimes,k}$.

It remains to show that
\[
\varphi(x)=\left(d_{\varphi(G)}(\varphi(w_1),\varphi(x)),d_{\varphi(G)}(\varphi(w_2),\varphi(x)),\dots,d_{\varphi(G)}(\varphi(w_k),\varphi(x))\right)
\]
for all $x\in V(G)$.  We claim that $d_{\varphi(G)}(\varphi(w_i),\varphi(x))=d_G(w_i,x)$ for all $i\in\{1,\dots,k\}$ and all $x\in V(G)$, from which the desired statement follows.  Let $i\in\{1,\dots,k\}$ and let $x\in V(G)$.  First of all, since $\varphi(G)$ obviously contains a copy of $G$ as a subgraph, we must have $d_G(w_i,x)\geq d_{\varphi(G)}(\varphi(w_i),\varphi(x))$.  On the other hand, since $\varphi(G)$ is a subgraph of $P_{D+1}^{\boxtimes,k}$, we have from Fact~\ref{distanceinstrongproducts} that
\begin{align*}
d_{\varphi(G)}(\varphi(w_i),\varphi(x))&=\max\{|d_G(w_j,x)-d_G(w_j,w_i)|\colon\ 1\leq j\leq k\}\\
&\geq |d_G(w_i,x)-d_G(w_i,w_i)|\\
&=d_G(w_i,x).
\end{align*}
We conclude that $d_{\varphi(G)}(\varphi(w_i),\varphi(x))=d_G(w_i,x)$, which completes the proof of the lemma.
\end{proof}

We now proceed with the proof of the main result of this section.

\begin{proof}[Proof of Theorem~\ref{Correspondence}]
$(\Rightarrow)$ Suppose that $W$ is a resolving set for some graph $H\in\super(G)$.  By Lemma \ref{charlemma1}, there is a $W$-resolved embedding $\varphi$ of $H$ in $P_{D+1}^{\boxtimes,k}$.  We show that $\varphi$ is also a $W$-resolved embedding of $G$ in $P_{D+1}^{\boxtimes,k}$.  By definition, we must have
\[
\varphi(x) =\left(d_{\varphi(H)}(\varphi(w_1), \varphi(x)), \ldots, d_{\varphi(H)}(\varphi(w_k), \varphi(x))\right)
\]
for all $x\in V(H).$
Since $G$ is a spanning subgraph of $H$, we have $V(G)=V(H)$, and the map $\varphi$ is also an embedding of $G$ in $P_{D+1}^{\boxtimes,k}$; if $xy\in E(G)$, then $xy\in E(H)$, and in turn $\varphi(x)\varphi(y)\in E\left(P_{D+1}^{\boxtimes,k}\right)$.   Further, we clearly have $\varphi(G)=\varphi(H)$, so
\begin{align*}
\varphi(x) &=\left(d_{\varphi(H)}(\varphi(w_1), \varphi(x)), \ldots, d_{\varphi(H)}(\varphi(w_k), \varphi(x))\right)\\
&=\left(d_{\varphi(G)}(\varphi(w_1), \varphi(x)), \ldots, d_{\varphi(G)}(\varphi(w_k), \varphi(x))\right).
\end{align*}
We conclude that $\varphi$ is a $W$-resolved embedding of $G$ in $P_{D+1}^{\boxtimes,k}$.

\smallskip

\noindent
$(\Leftarrow)$ Let $\varphi$ be a $W$-resolved embedding of $G$ in $P_{D+1}^{\boxtimes,k}$. From the definition of $W$-resolved embedding, we know that $\varphi$ is injective, and that
\[
\varphi(x)=\left(d_{\varphi(G)}(\varphi(w_1),\varphi(x)),d_{\varphi(G)}(\varphi(w_2),\varphi(x)),\dots,d_{\varphi(G)}(\varphi(w_k),\varphi(x))\right)
\]
for all $x\in V(G)$.  Therefore, every vertex $\varphi(x)\in V(\varphi(G))$ is uniquely determined by its distances in $\varphi(G)$ to members of the set $\varphi(W)=\{\varphi(w_1),\dots,\varphi(w_k)\}$, i.e., the set $\varphi(W)$ is a resolving set for $\varphi(G)$.
It follows that $W$ is a resolving set for the corresponding graph $H\in\super(G)$ with edge set $E(H)=\{xy\colon\ \varphi(x)\varphi(y)\in E(\varphi(G))\}$.
\end{proof}

\subsection{Applications}

We now discuss some applications of Theorem~\ref{Correspondence} and Lemma~\ref{charlemma1}.  First of all, Theorem~\ref{Correspondence} helped us to find trees of arbitrarily high metric dimension that have threshold dimension $2$. Let $L_{3n}$ be the tree obtained from the path $P_n$ by attaching two leaves to each vertex of $P_n$.  By Theorem~\ref{Slater}, we have $\beta(L_{3n})=n$. Figure~\ref{tree1z2} illustrates a $W$-resolved embedding of $L_{3n}$ in $P_{n+1}\boxtimes P_{n+1}$, where $|W|=2$, so we conclude that $\tau(L_{3n})=2$.  Thus, we have the following.

\begin{proposition}\label{LargeMetricSmallThreshold}
For every integer $b>2$, there is a tree $T$ with $\beta(T)=b$ and $\tau(T)=2.$
\end{proposition}

\begin{figure}[t]%tree1z2a,b,1z2
	
	\centering
	\begin{subfigure}[t]{0.49\textwidth}
		\centering
\begin{tikzpicture}[scale=0.75]
  \foreach \x in {0,...,7}
    \foreach \y in {0,...,7}
       {
       \vertex[opacity=0.25]  (\x\y) at (\x,\y) {};
       }
  \foreach \x in {0,1,2,3,5,6}
    \foreach \y in {0,1,2,3,5,6}
    {
    \pgfmathtruncatemacro{\a}{\x+1}
    \pgfmathtruncatemacro{\b}{\y+1}
    \path[opacity=0.25]
    (\x\y) edge (\a\b)
    (\x\b) edge (\a\y);
    }
\foreach \x in {0,...,3,5,6}
    \foreach \y in {0,...,4,6,7}
    {
    \pgfmathtruncatemacro{\a}{\x+1}
    \path[opacity=0.25]
    (\x\y) edge (\a\y)
    (\y\x) edge (\y\a);
    }
\foreach \x in {0,...,3,5,6}
{
\pgfmathtruncatemacro{\z}{\x+1}
\path[opacity=0.25]
(\x5) edge (\z5)
(5\x) edge (5\z)
;
}

\foreach \x in {0,...,4}
{
\node[left] at (0,\x) {\scriptsize \x};
\node[below] at (\x,0) {\scriptsize \x};
}

\node[left] at (0,5) {\scriptsize $n-2$};
\node[left] at (0,6) {\scriptsize $n-1$};
\node[left] at (0,7) {\scriptsize $n$};
\node[below] at (5,0) {\scriptsize $n-2$};
\node[below] at (6,0) {\scriptsize $n-1$};
\node[below] at (7,0) {\scriptsize $n$};

\foreach \x in {0,...,3,5,6}
{
\pgfmathtruncatemacro{\y}{\x+1}
\vertex (\x\y) at (\x,\y) {};
\vertex (\y\x) at (\y,\x) {};
\vertex (\y\y) at (\y,\y) {};
}
\vertex (55) at (5,5) {};
\foreach \x in {0,...,3,5,6}
{
\pgfmathtruncatemacro{\y}{\x+1}
\path[line width=1.5pt]
(\x\y) edge (\y\y)
(\y\x) edge (\y\y)
%(\x\y) edge (\y\x)
;
}
\foreach \x in {0,1,2,5}
{
\pgfmathtruncatemacro{\y}{\x+1}
\pgfmathtruncatemacro{\z}{\x+2}
\path[line width=1.5pt]
%(\x\y) edge (\y\z)
%(\y\x) edge (\z\y)
(\y\y) edge (\z\z)
%(\y\y) edge (\y\z)
%(\y\y) edge (\z\y)
;
}
\path[line width=1.5pt]
%(55) edge (56)
%(55) edge (65)
(55) edge (66);
%\foreach \x in {0,...,3}
%{
%\pgfmathtruncatemacro{\a}{14+\x}
%\pgfmathtruncatemacro{\z}{7+\x}
%\path[line width=1.5pt]
%(\z) edge (\a)
%(\a) edge (\x)
%(\z) edge (\x);
%}
%\foreach \x in {0,1,2,3,7,8,9,10,14,15,16,17}
%{
%\pgfmathtruncatemacro{\y}{\x+1}
%\path[line width=1.5pt]
%(\x) edge (\y);
%}
%\vertex (a) at (6,6) {};
%\path[line width=1.5pt]
%(a) edge (20);
\whitevertex (b) at (0,1) {};
\whitevertex (c) at (1,0) {};

\node at (2,4.5) {\rotatebox[origin=c]{90}{$\cdots$}};
\node at (4.5,2) {$\cdots$};
\node at (6,4.5) {\rotatebox[origin=c]{90}{$\cdots$}};
\node at (4.5,6) {$\cdots$};
\node at (4.5,4.5) {\rotatebox[origin=c]{45}{$\cdots$}};
%\path[line width=1.5pt]
%;
\end{tikzpicture}
		\caption{The embedding $\varphi$ of $L_{3n}$ in $P_n\boxtimes P_n$.}
		\label{tree1z2a}

	\end{subfigure}	
	\hfill
	\centering
	\begin{subfigure}[t]{0.49\textwidth}
		\centering
\begin{tikzpicture}[scale=0.75]
  \foreach \x in {0,...,7}
    \foreach \y in {0,...,7}
       {
       \vertex[opacity=0.25]  (\x\y) at (\x,\y) {};
       }
  \foreach \x in {0,1,2,3,5,6}
    \foreach \y in {0,1,2,3,5,6}
    {
    \pgfmathtruncatemacro{\a}{\x+1}
    \pgfmathtruncatemacro{\b}{\y+1}
    \path[opacity=0.25]
    (\x\y) edge (\a\b)
    (\x\b) edge (\a\y);
    }
\foreach \x in {0,...,3,5,6}
    \foreach \y in {0,...,4,6,7}
    {
    \pgfmathtruncatemacro{\a}{\x+1}
    \path[opacity=0.25]
    (\x\y) edge (\a\y)
    (\y\x) edge (\y\a);
    }
\foreach \x in {0,...,3,5,6}
{
\pgfmathtruncatemacro{\z}{\x+1}
\path[opacity=0.25]
(\x5) edge (\z5)
(5\x) edge (5\z)
;
}

\foreach \x in {0,...,4}
{
\node[left] at (0,\x) {\scriptsize \x};
\node[below] at (\x,0) {\scriptsize \x};
}

\node[left] at (0,5) {\scriptsize $n-2$};
\node[left] at (0,6) {\scriptsize $n-1$};
\node[left] at (0,7) {\scriptsize $n$};
\node[below] at (5,0) {\scriptsize $n-2$};
\node[below] at (6,0) {\scriptsize $n-1$};
\node[below] at (7,0) {\scriptsize $n$};

\foreach \x in {0,...,3,5,6}
{
\pgfmathtruncatemacro{\y}{\x+1}
\vertex (\x\y) at (\x,\y) {};
\vertex (\y\x) at (\y,\x) {};
\vertex (\y\y) at (\y,\y) {};
}
\vertex (55) at (5,5) {};
\foreach \x in {0,...,3,5,6}
{
\pgfmathtruncatemacro{\y}{\x+1}
\path[line width=1.5pt]
(\x\y) edge (\y\y)
(\y\x) edge (\y\y)
(\x\y) edge (\y\x);
}
\foreach \x in {0,1,2,5}
{
\pgfmathtruncatemacro{\y}{\x+1}
\pgfmathtruncatemacro{\z}{\x+2}
\path[line width=1.5pt]
(\x\y) edge (\y\z)
(\y\x) edge (\z\y)
(\y\y) edge (\z\z)
(\y\y) edge (\y\z)
(\y\y) edge (\z\y);
}
\path[line width=1.5pt]
(55) edge (56)
(55) edge (65)
(55) edge (66);
%\foreach \x in {0,...,3}
%{
%\pgfmathtruncatemacro{\a}{14+\x}
%\pgfmathtruncatemacro{\z}{7+\x}
%\path[line width=1.5pt]
%(\z) edge (\a)
%(\a) edge (\x)
%(\z) edge (\x);
%}
%\foreach \x in {0,1,2,3,7,8,9,10,14,15,16,17}
%{
%\pgfmathtruncatemacro{\y}{\x+1}
%\path[line width=1.5pt]
%(\x) edge (\y);
%}
%\vertex (a) at (6,6) {};
%\path[line width=1.5pt]
%(a) edge (20);
\whitevertex (b) at (0,1) {};
\whitevertex (c) at (1,0) {};

\node at (2,4.5) {\rotatebox[origin=c]{90}{$\cdots$}};
\node at (4.5,2) {$\cdots$};
\node at (6,4.5) {\rotatebox[origin=c]{90}{$\cdots$}};
\node at (4.5,6) {$\cdots$};
\node at (4.5,4.5) {\rotatebox[origin=c]{45}{$\cdots$}};
%\path[line width=1.5pt]
%;
\end{tikzpicture}
		\caption{The graph $\varphi(L_{3n})$.}
		\label{tree1z2b}
	\end{subfigure}	
	\caption{A $W$-resolved embedding $\varphi$ of the graph $L_{3n}$ in $P_n\boxtimes P_n$.  The vertices of $W$ are coloured white.}
	\label{tree1z2}
\end{figure}
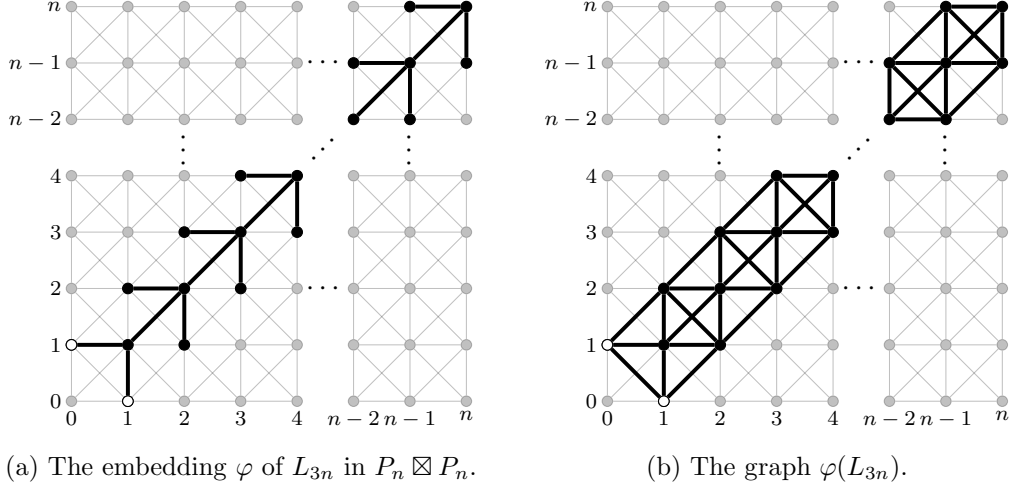

The following corollary of Proposition~\ref{LargeMetricSmallThreshold} is a strengthening of a result of Chartrand et al.~\cite{Chartrandetal2000}, which states that for every $\epsilon >0$, there exists a graph $H$ and a connected subgraph $G$ of $H$ such that $\beta(H)/\beta(G)<\epsilon$.

\begin{corollary}
For every $\epsilon > 0$, there exists a graph $G$ such that $\tau(G)/ \beta(G) <\epsilon$, i.e., there exists a graph $H$ and a connected spanning subgraph $G$ of $H$ such that $\beta(H)/\beta(G)<\epsilon$.
\end{corollary}

While the embedding defined in the proof of Lemma~\ref{charlemma1} is simple, it can be used to establish several known results.  We include short, intuitive proofs of these results to emphasize the usefulness of Lemma~\ref{charlemma1}.

\begin{corollary} \label{dimsg}
Let $G$ be a connected graph of order $n$ and diameter $D_G$. If $G$ is a subgraph of $H$, then $\beta(H) \geq \ceeil{\log_{D_G+1}n}$.
\end{corollary}

\begin{proof}
Let $D_H$ be the diameter of $H$ and suppose $\beta(H)=b$. Then there is a set $W\subseteq V(H)$ of cardinality $b$ that resolves $H$.  By Lemma~\ref{charlemma1}, there is a $W$-resolved embedding of $H$ in $P_{D_H+1}^{\boxtimes,b}$. So $G$ is also embedded in this strong product.  

For every $1 \le i \le b$, let $m_i$ denote the minimum $i^{th}$ co-ordinate and $M_i$ the maximum $i^{th}$ co-ordinate among all vertices of $G$ in this embedding. Since $G$ has diameter $D_G$, we see that $M_i \le \min \{m_i+D_G, D_H\}$. For $1 \le i \le b$, let $P_i$ be the subpath of $P$ induced by the vertices $m_i, m_i+1, \ldots, M_i$. Then $G$ is embedded in the strong product  $P_1 \boxtimes P_2 \boxtimes \ldots \boxtimes P_b$ of paths each of order at most $D_G+1$. It follows that $(D_G+1)^b \geq n$.
\end{proof}

The following result was proven by Hernando et al.~\cite{Hernandoetal2010}.  The specific case $b=2$ was also proven by Javaid et al.~\cite{Javaidteal2008} and Sudhakara et al.~\cite{Sudhakaraetal2009}.  It has often been used to establish a lower bound on the metric dimension of a given graph.

\begin{corollary}\label{2dim}
Let $G$ be a connected graph of diameter $D$ with resolving set $W=\{w_1, \ldots, w_b\}$.  Then for every $1 \leq i \leq b$, and every $1 \leq k \leq D$, we have $|N_k(w_i)| \leq (2k+1)^{b-1}$.
\end{corollary}

\begin{proof}
By symmetry, it suffices to show that $|N_k(w_1)| \leq (2k+1)^{b-1}$ for all $1\leq k\leq D$. Let $\varphi$ be the $W$-resolved embedding of $G$ in $P_{D+1}^{\boxtimes,b}$ defined in the proof of Lemma~\ref{charlemma1} as follows:
\[
\varphi(x)=\left(d_G(w_1,x),d_G(w_2,x),\dots,d_G(w_b,x)\right).
\]
Suppose that $x \in N_k(w_1)$, or equivalently, that $d_G(w_1,x)=k$.  Then for all $2\leq i\leq b$, by the reverse triangle inequality, we have
\[
|d_G(w_i,w_1)-d_G(w_i,x)|\leq d_G(w_1,x)=k.
\]
It follows that there are at most $2k+1$ possible values for the $i$th coordinate $d_G(w_i,x)$ of $\varphi(x)$, for all $2\leq i\leq b$.  We conclude that $|N_k(w_1)| \leq (2k+1)^{b-1}$.
\end{proof}

\subsection{Comparing the threshold dimension and the strong isometric dimension}

By Corollary~\ref{ThresholdEmbedding}, the threshold dimension of a graph $G$ is the smallest integer $k$ for which there is a $W$-resolved embedding of $G$ in $P^{\boxtimes,k}$ for some set $W\subseteq V(G)$ of cardinality $k$ and some sufficiently large path $P$.  Thus, it is natural to ask how the threshold dimension of $G$ compares to the \emph{strong isometric dimension} of $G$, denoted $\mbox{sdim}(G)$, and defined as the minimum integer $k$ such that there is an isometric embedding of $G$ in $P^{\boxtimes,k}$ for some path $P$.  See~\cite[Chapter 15]{ProductHandbook} for a brief survey of results on the strong isometric dimension.

Given a graph $G$ and a set $W\subseteq V(G)$ of cardinality $k$, a $W$-resolved embedding of $G$ in $P^{\boxtimes,k}$ need not be isometric.  See, for example, the embeddings shown in Figure~\ref{ResolvedEmbedding1} and Figure~\ref{ResolvedEmbedding2}, neither of which is isometric.
Note that if $W$ is a resolving set for $G$, then the $W$-resolved embedding of $G$ defined in the proof of Lemma~\ref{charlemma1} preserves the distance between vertices $w\in W$ and $v\in V(G)$, but it does not necessarily preserve the distance between every pair of vertices of $V(G)-W$.  See, for example, the embedding shown in Figure~\ref{ResolvedEmbedding1}.

It is also easy to see that an isometric embedding of $G$ in $P^{\boxtimes,k}$ need not be a $W$-resolved embedding of $G$ for any set $W\subseteq V(G)$ of cardinality $k$.  Take, for example, any embedding $\varphi$ of the complete graph $K_4$ in $P_2\boxtimes P_2$, which is clearly isometric (as it must be an isomorphism).  Since we must have $\varphi(x)=(0,0)$ for some $x\in V(K_4)$, we see immediately that $\varphi$ is not a $W$-resolved embedding for any set $W\subseteq V(K_4)$.

Indeed, there is no general order relation between $\tau(G)$ and $\mbox{sdim}(G)$.  We have already observed that $\tau(K_n)=n-1$, while it is well-known that $\mbox{sdim}(K_n)=\ceeil{\log_2(n)}$ (see~\cite[Theorem~15.4]{ProductHandbook}).  On the other hand, for $n\geq 4$, the cycle $C_n$ has $\tau(C_n)=\beta(C_n)=2$, and $\mbox{sdim}(C_n)=\ceeil{n/2}$ (see~\cite{FitzpatrickNowakowski2000}).  Overall, it appears that the threshold dimension and the strong isometric dimension of a graph are rather distinct measures.

\section{The threshold dimension and reducibility of trees}\label{trees}

In this section we focus on the threshold dimension of trees.  We begin by establishing a sharp upper bound on the threshold dimension of every tree of order $n$.  For every positive integer $n$, we define $g(n)$ to be the least nonnegative integer $d$ such that $2^d+d\geq n$.  Note that $g(n)\leq \log_2(n)$ for all $n$.

\begin{theorem} \label{starmeth}
Let $T$ be a tree of order $n\geq 1$.  Then $\tau(T) \leq g(n)$, and this bound is sharp.
\end{theorem}

\begin{proof}
If $\beta(T)\leq g(n)$, then the result follows immediately.  So suppose that $\beta(T)\geq g(n)$.  By Theorem~\ref{Slater}, it follows that $T$ must have more than $g(n)$ leaves.  Let $W$ be a set of exactly $g(n)$ leaves of $T$.  We show that there is a graph $H\in \super(G)$ for which every vertex in the set $V(T)-W$ has a distinct $W$-neighbourhood in $H$.  It follows that $W$ is a resolving set for $H$,
and hence $\tau(T)\leq |W|=g(n)$, as desired.
	
Let $X=V(T)-W$.  By the definition of $g(n)$, we have $n\leq 2^{|W|}+|W|,$ and this means that
\[
|X|=n-|W|\leq 2^{|W|}.
\]
Thus, we have $|X|\leq |\mathcal{P}(W)|$, where $\mathcal{P}(W)$ is the power set of $W$.  We think of the set $\mathcal{P}(W)$ as the set of all possible $W$-neighbourhoods of vertices in $X$.  Let $X_1=\left\{x\in V(T)-W\colon\ N_W^T(x)\neq \emptyset\right\}$, and $X_2=\left\{x\in V(T)-W\colon\ N_W^T(x)=\emptyset\right\}$.  Since $W$ is a set of leaves of $T$, we see that if $x,y\in X_1$, then $N_W^T(x)\cap N_W^T(y)=\emptyset$, and hence $N_W^T(x)$ and $N_W^T(y)$ are distinct.  Let $\mathcal{M}=\left\{N_W^T(x)\colon\ x\in X_1\right\}$, and let $\mathcal{N}=\mathcal{P}(W)-\mathcal{M}$.  Since $|X|\leq \mathcal{P}(W)$, we can assign to every vertex $x\in X_2$ a unique subset $S_x\in \mathcal{N}$.  For every $x\in X_2$, let $E_x=\{xw\colon\ w\in S_x\}$.  Let $E=\cup_{x\in X_2}E_x$.  Then in the graph $G+E$, we have $N_W^{G+E}(x)=S_x$ for every $x\in X_2$.  Since $N_W^{G+E}(x)=N_W^{T}(x)$ for every $x\in X_1$, we see that every vertex $x\in X$ has a unique $W$-neighbourhood in $G+E$.  We conclude that $W$ is a resolving set for $G+E$, and hence $\tau(T)\leq g(n)$.

Finally, we illustrate sharpness of the bound by showing that $\tau(K_{1,n-1})\geq g(n)$. Let $H\in\super(K_{1,n-1})$.  Then $H$ is a connected graph of diameter $2$. By a result of Chartrand et al.~\cite[Theorem 1]{Chartrandetal2000}, we have $\beta(H)\geq g(n)$.  Since $H$ was an arbitrary graph in $\super(K_{1,n-1})$, we conclude that $\tau(K_{1,n-1})\geq g(n)$.
\end{proof}

We now show that every tree with dimension at least $3$ is reducible. In fact, we show that for every tree with dimension at least $3$, there is a single edge whose addition to the tree decreases the dimension.  To aid us in our discussions, we introduce some more terminology. Before proceeding, we encourage the reader to revisit the terminology introduced immediately before Theorem~\ref{Slater}.

Let $T$ be a tree that is not isomorphic to a path.  We say that a limb $L$ of $T$ is \emph{adjacent} to a vertex $v$ in $T$ if some endnode of $L$ is adjacent to $v$ in $T$, i.e.,~if the leaf of $T$ contained in $L$ is a terminal vertex of $v$.  The \emph{core} of $T$, denoted $c(T)$, is the tree obtained from $T$ by deleting all of its limbs.  Note that every leaf of $c(T)$ must be adjacent to at least two limbs of $T$, and hence must have terminal degree at least two in $T$.  For a vertex $x$ of $T$, let $C_1, ..., C_k$ be the components of  $T-x$. Then the \emph{branches} of $T$ at $x$ are the induced subgraphs $B_i=T[V(C_i)\cup \{x\}]$ for $i\in\{1,\dots,k\}$.

\begin{theorem} \label{tree}
Let $T$ be a tree with $\beta(T)\geq 3$. Then there exists an edge $e\in E\left(\overline{T}\right)$ such that $\beta(T+e)<\beta(T)$.
\end{theorem}

\begin{proof}
We consider two cases.
	
	\smallskip
	
	\noindent {\bf Case 1}: There is an exterior major vertex $v$ of $T$ with terminal degree at least $3$.
	
	\smallskip
	
	\noindent Let $B_1, ..., B_k$ be the branches of $T$ at $v$. Without loss of generality, we may assume that $B_1$ and $B_2$ are paths.  Let $T'=T[V(B_3)\cup \cdots\cup V(B_k)]$.  Let $v_i$ be the neighbour of $v$ in $B_i$ for all $i\in\{1,2\}$.  By Theorem~\ref{Slater}, there is a basis $W$ of $T$ that contains $v_1$ and $v_2$.  Since $\beta(T)\geq 3$, the basis $W$ must also contain some vertex $u$ of $T'$. Let $F = T+v_1v_2$.  See Figure~\ref{3xt} for an illustration of $T$ (and $F$), with the new edge $v_1v_2$ drawn as a dashed line.  We claim that $W'=W-\{v_1\}$ resolves $F$.
	
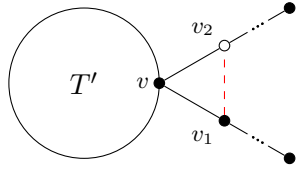
\begin{figure}
\centering
		\begin{tikzpicture}
		\draw (-1,0) circle (1);
		\node at (-1,0) {$T'$};
		\vertex (0) at (0:0) {};
		\node[left] at (0) {\footnotesize $v$};
		\vertex (1) at (-30:1) {};
		\node[below left] at (1) {\footnotesize $v_1$};
		\vertex (2) at (-30:2) {};
		\whitevertex (3) at (30:1) {};
		\node[above left] at (3) {\footnotesize $v_2$};
		\vertex (4) at (30:2) {};
		\path
		(0) edge (1)
		(0) edge (3)	
		;
		\path[dashed,red]
		(1) edge (3);
		\draw[middlearrow={ellipsis}{0.7}] (1) -- (2);
		\draw[middlearrow={ellipsis}{0.7}] (3) -- (4);
		\end{tikzpicture}

		\caption{The tree $T$ in Case 1 of the proof of Theorem~\ref{tree}}
		\label{3xt}
\end{figure}
	
	Let $x$ and $y$ be distinct vertices of $F$.  We will show that some vertex of $W'$ resolves $x$ and $y$.  Firstly, if $x=v$, then $x$ lies on either a shortest $u-y$ path (if $y\in V[B_1]\cup V[B_2]$, or a shortest $v_2-y$ path (if $y\in V[T']$). So either $u$ or $v_2$ resolves $x$ and $y$. Thus, in all subsequent cases we may assume that neither $x$ nor $y$ is equal to $v$.
	
	Suppose that $x\in V(B_i)$ and $y \in V(B_j)$.  Without loss of generality, we may assume that $i\leq j$.  If $i>1$, then $x$ and $y$ lie in the graph $T''=F[V(B_2)\cup V(T')]=T[V(B_2)\cup V(T')]$.  Evidently, the tree $T''$ is an isometric subgraph of $F$.  By Theorem~\ref{Slater}, the set $W'$ is a resolving set for $T''$, and hence $W'$ resolves $x$ and $y$ in $F$.
	
So we may assume that $i=1$.  If $j=1$, then $x$ and $y$ are clearly resolved by $v_2$.	If $j=2$, observe that if $d_F(u,x) = d_F(u,y)$, then $d_F(v_2,y)<d_F(v_2,x)$, so $x$ and $y$ are resolved by either $u$ or $v_2$ in $F$.  Finally, if $j\geq 3$, then we claim again that $x$ and $y$ are resolved by either $u$ or $v_2$ in $F$.  Suppose otherwise that we have both $d_F(u,x)=d_F(u,y)$ and $d_F(v_2,x)=d_F(v_2,y)$.  Then we have
\[
d_F(v,x)=d_F(v_2,x)=d_F(v_2,y)=d_F(v,y)+1.
\]
It follows that
\[
d_F(u,y)=d_F(u,x)=d_F(u,v)+d_F(v,x)=d_F(u,v)+d_F(v,y)+1,
\]
and this is impossible since we must have $d_F(u,v)+d_F(v,y)\geq d_F(u,y)$.
	
We conclude that $W'$ resolves $F$.  Since $|W'|=|W|-1$, this completes the proof in this case.

\smallskip

	\noindent {\bf Case 2}: The terminal degree of every exterior major vertex of $T$ is at most 2.
	
	\smallskip

\noindent	In this case, every exterior major vertex contributes at most 1 to the dimension of $T$.  Since $\beta(T)\geq 3$, the tree $T$ must have at least three vertices of terminal degree $2$, all of which must belong to the core $c(T)$ of $T$.  Note that every leaf of $c(T)$ must have terminal degree exactly $2$.  If $c(T)$ is not a path, then consider the core $c(c(T))$ of $c(T)$.  If $c(c(T))$ is non-trivial, let $v$ be a leaf of $c(c(T))$; otherwise let $v$ be the unique vertex of $c(c(T))$. Then in $c(T)$, the vertex $v$ is adjacent to at least two limbs of $c(T)$.  Let $v_1$ and $v_2$ be the leaves of $c(T)$ at the ends of two of these limbs.  If $c(T)$ is a path, then let $v_1$ and $v_2$ be the leaves of $c(T)$, and let $v$ be any other vertex in $c(T)$.  We have two subcases.
	
	\smallskip
	
	\noindent \textbf{Subcase 2.1:} No internal vertex of the $v_1v_2$-path has terminal degree $2$ in $T$.
	
	\smallskip
	
	\noindent Let $B_1,B_2,\dots,B_k$ be the branches of $T$ at $v$, and assume that $v_1\in V(B_1)$ and $v_2\in V(B_2)$.  Let $T'=T[V(B_3)\cup\cdots\cup V(B_k)]$.  For $i\in \{1,2\}$, let $v'_{i}$ be a neighbour of $v_i$ that lies on a limb of $T$ (see Figure~\ref{2xt}).

\begin{figure}
	\centering
			\begin{tikzpicture}
		\draw (-1,0) circle (1);
		\node at (-1,0) {$T'$};
		\vertex (0) at (0:0) {};
		\node[left] at (0) {\footnotesize $v$};
		\vertex (3) at (30:1) {};
		\vertex (4) at (30:3) {};
		\vertex (5) at (-30:1) {};
		\vertex (6) at (-30:2) {};
		\node[below left] at (6) {\footnotesize $v_1$};
		\vertex (7) at (30:4) {};
		\node[above left] at (7) {\footnotesize $v_2$};
		\vertex (8) at (30:5) {};
%		\node[above left] at (8) {\footnotesize $v_{22}$};
		\whitevertex (9) at (1.73*5/2,1.5) {};
		\node[below left] at (9) {\footnotesize $v'_2$};
		\vertex (10) at (1.73*3,1) {};
		\vertex (11) at (-30:3) {};
%		\node[below left] at (11) {\footnotesize $v_{12}$};
		\vertex (12) at (-30:4) {};
		\vertex (13) at (1.73*3/2,-0.5) {};
		\node[below right] at (13) {\footnotesize $v'_{1}$};
		\vertex (14) at (1.73*2,0) {};
		\vertex (15) at (30:6) {};
		\path
		(0) edge (3)
		(0) edge (5)
		(7) edge (9)
		(7) edge (8)
		(6) edge (11)
		(6) edge (13)		
		;
		\path[dashed,red]
		(13) edge (4);
		\path[-...] (3) edge (30:2.1);
		\draw[middlearrow={ellipsis}{0.6}] (3) -- (4);
		\draw[middlearrow={ellipsis}{0.7}] (5) -- (6);
		\draw[middlearrow={ellipsis}{0.7}] (8) -- (15);
		\draw[middlearrow={ellipsis}{0.7}] (4) -- (7);
		\draw[middlearrow={ellipsis}{0.7}] (9) -- (10);
		\draw[middlearrow={ellipsis}{0.7}] (13) -- (14);
		\draw[middlearrow={ellipsis}{0.7}] (11) -- (12);
		\draw[|-|] (0.25,-0.02) -- node[above]{\footnotesize $d_1$}(3/2*1.73-0.2,-0.45);
		\draw[|-|,yshift=6pt] (0.1,0.1) -- node[above]{\footnotesize $d_1$}(3/2*1.73-0.1,1.5);
		\end{tikzpicture}
	\caption{The tree $T$ in Subcase 2.1 of the proof of Theorem~\ref{tree}.  Note that every internal vertex of both the $vv_1$-path and the $vv_2$-path may be adjacent to a single limb -- these limbs are not drawn.}
	\label{2xt}
\end{figure}
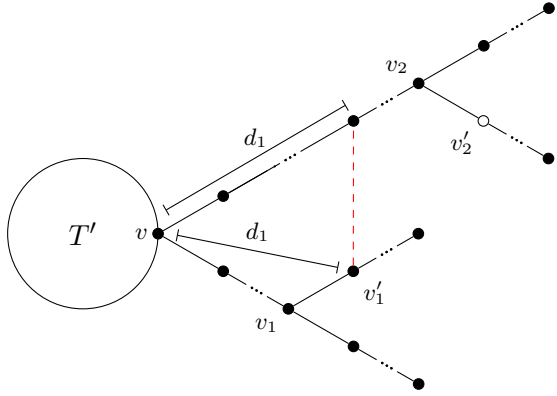

By Theorem~\ref{Slater}, there is a basis $W$ for $T$ that contains $v'_{1}$ and $v'_{2}$.  Since $\beta(T)\geq 3$, there must be some vertex $u\in W$ that lies in $T'-v$.  Let $d_T(v, v'_i)= d_i$ for $i\in\{1,2\}$. We may assume that $d_1  \leq d_2$. Let $(v =) z_0z_1\cdots z_{d_2}(=v'_{2})$ be the $vv'_{2}$-path in $T$. Let $F$ be obtained from $T$ by joining $v'_{1}$ and $z_{d_1}$, i.e., $F =T+v'_{1}z_{d_1}$. The new edge is shown as a dashed line in Figure \ref{2xt}.  Let $W'=W-\{v'_{1}\}$.  We claim that $W'$ resolves $F$.

Let $x$ and $y$ be distinct vertices in $T$.  We will show that $x$ and $y$ are resolved by some vertex of $W'$.  By an argument as in Case 1, we may assume that neither $x$ nor $y$ is equal to $v$.  Suppose that $x \in B_i$ and  $y \in B_j$, and without loss of generality, assume that $i\leq j$.  The case that $i>1$ is handled just as in Case 1, so we may assume that $i=1$.

First suppose that $j=1$.  We claim that $x$ and $y$ are resolved in $F$ by either $u$ or $v'_{2}$.  Suppose towards a contradiction that $d_F(u,x)=d_F(u,y)$ and $d_F(v'_{2},x)=d_F(v'_{2},y)$.  Then we have $d_F(v,x)=d_F(v,y)$ and $d_F(v'_1,x)=d_F(v'_1,y)$.  From the first of these two observations we see that $d_T(v,x)=d_F(v,x)=d_F(v,y)=d_T(v,y)$. So $v'_1$ must resolve $x$ and $y$ in $T$. However, from the second observation we see that $d_T(v'_1,x)=d_F(v'_1,x) =d_F(v'_1,y)=d_T(v'_1,y)$, contradicting the fact that $v'_1$ resolves $x$ and $y$ in $T$.
%Let $x'$ be the vertex closest to $x$ on the $vv'_1$-path of $T$ (possibly $x'=x$), and let $y'$ be the vertex closest to $y$ on the $vv'_1$-path of $T$ (possibly $y'=y$).  Then we have
%\begin{align*}
%d(v,x)+d(v'_1,x)&=\left[d(v,x')+d(x',x)\right]+\left[d(v'_1,x')+d(x',x)\right]\\
%&=d(v,v'_1)+2d(x',x),
%\end{align*}
%and
%\begin{align*}
%d(v,y)+d(v'_1,y)&=\left[d(v,y')+d(y',y)\right]+\left[d(v'_1,y')+d(y',x)\right]\\
%&=d(v,v'_1)+2d(y',y),
%\end{align*}
%Since $d(v,x)+d(v'_1,x)=d(v,y)+d(v'_1,y)$, we have $d(x',x)=d(y',y).$  However, since $d(v,x')+d(x',x)=d(v,x)=d(v,y)=d(v,y')+d(y',y)$, it follows that $d(v,x')=d(v,y')$. So $x'=y'$, and in turn, we have $x=y$.  This contradicts the assumption that $x$ and $y$ are distinct.

Suppose next that $j=2$.  Again, we claim that $x$ and $y$ are resolved in $F$ by either $u$ or $v'_2$.  Suppose otherwise that $d_F(u,x)=d_F(u,y)$ and $d_F(v'_2,x)=d_F(v'_2,y)$.  Let $x'$ be the vertex closest to $x$ on the $vv'_1$-path of $T$, and let $y'$ be the vertex closest to $y$ on the $vv'_2$-path of $T$.  Then we have
\begin{align*}
d_F(u,x)+d_F(v'_2,x)&=\left[d_F(u,v)+d_F(v,x')+d_F(x',x)\right]\\
&\hspace{0.5cm}+\left[d_F(v'_2,v'_1)+d_F(v'_1,x')+d_F(x',x)\right]\\
&=d_F(u,v)+2d_F(x',x)+\left[d_F(v,x')+d_F(x',v'_1)+d_F(v'_1,v'_2)\right]\\
&=d_F(u,v)+2d_F(x',x)+d_F(v,v'_2)+1,
\end{align*}
and
\begin{align*}
d_F(u,y)+d_F(v'_2,y)&=\left[d_F(u,v)+d_F(v,y')+d_F(y',y)\right]+\left[d_F(v'_2,y')+d_F(y',y)\right]\\
&=d_F(u,v)+2d_F(y',y)+\left[d_F(v,y')+d_F(v'_2,y')\right]\\
&=d_F(u,v)+2d_F(y',y)+d_F(v,v'_2).
\end{align*}
Since $d_F(u,x)+d_F(v'_2,x)=d_F(u,y)+d_F(v'_2,y)$, we obtain
\[
2d_F(x',x)+1=2d_F(y',y),
\]
which is impossible.

Finally, suppose that $j\geq 3$.  Again, we claim that $x$ and $y$ are resolved in $F$ by either $u$ or $v'_2$.  Suppose otherwise that $d_F(u,x)=d_F(u,y)$ and $d_F(v'_2,x)=d_F(v'_2,y)$.  Let $x'$ be the vertex closest to $x$ on the $vv'_1$-path of $T$, and let $y'$ be the first vertex that lies on both the $yv$-path and the $uv$-path in $T$.  Note that $y'\in T'$, and we may have $y'\in\{u,v,y\}$.  Then we have
\begin{align*}
d_F(u,x)+d_F(v'_2,x)&=\left[d_F(u,v)+d_F(v,x')+d_F(x',x)\right]\\
&\hspace{0.5cm}+\left[d_F(v'_2,v'_1)+d_F(v'_1,x')+d_F(x',x)\right]\\
&=d_F(u,v)+2d_F(x',x)+\left[d_F(v,x')+d_F(x',v'_1)+d_F(v'_1,v'_2)\right]\\
&=d_F(u,v)+2d_F(x',x)+d_F(v,v'_2)+1,
\end{align*}
and
\begin{align*}
d_F(u,y)+d_F(v'_2,y)&=\left[d_F(u,y')+d_F(y',y)\right]+\left[d_F(v'_2,v)+d_F(v,y')+d_F(y',y)\right]\\
&=d_F(u,v)+2d_F(y',y)+d_F(v,v'_2).
\end{align*}
Since $d_F(u,x)+d_F(v'_2,x)=d_F(u,y)+d_F(v'_2,y)$, we obtain
\[
2d_F(x',x)+1=2d_F(y',y),
\]
which is impossible.

	\smallskip

	\noindent \textbf{Subcase 2.2:} Some internal vertex of the $v_1v_2$-path has terminal degree $2$ in $T$.
	
	\smallskip
	
		\begin{figure}
		\centering
			\begin{tikzpicture}
		\draw (-1,0) circle (1);
		\node at (-1,0) {$T'$};
		\vertex (0) at (0:0) {};
		\node[left] at (0) {\footnotesize $w$};
		\vertex (1) at (1,-0.5) {};
		\node[above right] at (1) {\footnotesize $w'$};
		\vertex (2) at (2,-1) {};
		\vertex (3) at (1,0.5) {};
		\vertex (4) at (2,1) {};
		\node[above left] at (4) {\footnotesize $v_{2}$};
%		\vertex (5) at (1,-0.75) {};
%		\vertex (6) at (2,-1.5) {};
		\whitevertex (7) at (3,0.5) {};
		\node[below left] at (7) {\footnotesize $v'_{2}$};
		\vertex (8) at (4,0) {};
		\vertex (9) at (3,1.5) {};
		\vertex (10) at (4,2) {};
		\vertex (11) at (0.5,-1) {};
		\vertex (12) at (1,-2) {};
		\path
		(0) edge (1)
		(0) edge (3)
%		(0) edge (5)
		(4) edge (7)
		(4) edge (9)
		(0) edge (11)		
		;
		\path[dashed,red]
		(1) edge (3);
		\draw[middlearrow={ellipsis}{0.7}] (1) -- (2);
		\draw[middlearrow={ellipsis}{0.7}] (3) -- (4);
%		\path[-...] (5) edge (1.7,-1.275);
%		\path (1.7,-1.275) edge (6);
		\draw[middlearrow={ellipsis}{0.7}] (7) -- (8);
		\draw[middlearrow={ellipsis}{0.7}] (9) -- (10);
		\draw[middlearrow={ellipsis}{0.7}] (11) -- (12);
		\end{tikzpicture}
	\caption{The tree $T$ in Subcase 2.2 of the proof of Theorem~\ref{tree}.  Note that every internal vertex of the $wv_2$-path may be adjacent to a single limb -- these limbs are not drawn.}
	\label{2xtb}
\end{figure}
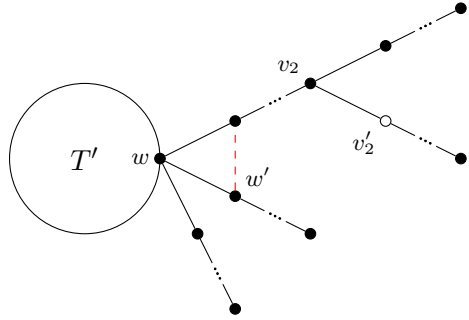
	
	\noindent
	Suppose without loss of generality that the $v_2v$-path contains a vertex of terminal degree $2$ in $T$ other than $v_2$.  Let $w$ be the first such vertex on the $v_2v$-path.  Let $v_2'$ be one of the neighbours of $v_2$ that lies on a limb of $T$, and let $w'$ be one of the neighbours of $w$ that lies on a limb of $T$.  Let $B_1,B_2,\dots,B_k$ be the branches of $T$ at $w$, such that $w'\in V(B_1)$,$v_2\in V(B_2)$, and $B_3$ is a path, and let $T'=T[V(B_4)\cup\cdots\cup V(B_k)]$.  By Theorem~\ref{Slater}, there is a basis $W$ of $T$ that contains $w'$, $v_2'$, and some vertex $u\in V(T')$.  Let $F$ be the graph obtained from $T$ by joining $w'$ and the neighbour of $w$ in $B_2$ (see Figure~\ref{2xtb}, where the new edge is shown as a dashed line).  We claim that $W'=W-\{w'\}$ is a resolving set for $F$.  This claim is established by the proof of Subcase 2.1, with $v$ and $v_1$ both replaced by $w$, and $v'_1$ replaced by $w'$.
\end{proof}

%The following corollary is immediate.
%
%\begin{corollary}
%	Let $T$ be a tree with $\beta(T)=3$. Then $\tau(T)=2$.
%\end{corollary}

\begin{remark}\label{TprimeRemark}
In Case 2 of the proof of Theorem~\ref{tree}, we assumed that $T$ had no vertices of terminal degree greater than $2$.  However, this assumption was not important to the proof of Subcase~2.1 or Subcase~2.2.  Indeed, the only fact about the subtree $T'$ used in the proof of either subcase is that $T'$ contains some vertex in the chosen resolving set for $T$.
\end{remark}

If a tree $T$ is reducible, then we necessarily have $\beta(T)\geq 3$, and hence by Theorem~\ref{tree}, there is a single edge whose addition to $T$ produces a graph with lower dimension than $T$.  We remark that this does not hold for graphs in general.  For example, consider the graph $G$ shown in Figure~\ref{k10a}. One can show that $\beta(G)=3$, and that for every edge $e\in E\left(\overline{G}\right)$, we have $\beta(G+e)=3$.  However, the graph $G$ is reducible.  There is a graph $H\in \super(G)$ obtained from $G$ by adding two particular edges (see Figure~\ref{k10b}) such that $\beta(H)=2$.

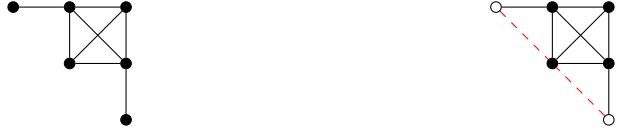
\begin{figure}[h] %notoneedge
	\centering
	\begin{subfigure}[t]{0.45\textwidth}
	\centering
	\begin{tikzpicture}[scale=0.75]	
	\begin{scope}[rotate=-90]
	\vertex (x1) at (0,1) {};
	\vertex (x2) at (0,0) {};
	\vertex (x3) at (1,0) {};
	\vertex (x4) at (1,1) {};
	\vertex (v1) at (2,1) {};
	\vertex (v2) at (0,-1) {};
	
	\path
	(x1) edge (x2)
	(x1) edge (x3)
	(x1) edge (x4)
	(x2) edge (x3)
	(x2) edge (x4)
	(x3) edge (x4)
	(v2) edge (x2)
	(v1) edge (x4)
	;	
	
%	\path [dashed]
%	(v1) edge (x3)
%	(v2) edge (x3);
	\end{scope}
	\end{tikzpicture}
	\caption{A graph $G$ with $\beta(G)=3$.}
	\label{k10a}
	\end{subfigure}
	\begin{subfigure}[t]{0.45\textwidth}
	\centering
	\begin{tikzpicture}[scale=0.75]
	\begin{scope}[rotate=-90]
	\vertex (x1) at (0,1) {};
	\vertex (x2) at (0,0) {};
	\vertex (x3) at (1,0) {};
	\vertex (x4) at (1,1) {};
	\whitevertex (v1) at (2,1) {};
	\whitevertex (v2) at (0,-1) {};
	
	\path
	(x1) edge (x2)
	(x1) edge (x3)
	(x1) edge (x4)
	(x2) edge (x3)
	(x2) edge (x4)
	(x3) edge (x4)
	(v2) edge (x2)
	(v1) edge (x4)
	;	
	
	\path [red,dashed]
	(v1) edge (x3)
	(v2) edge (x3);
	\end{scope}
	\end{tikzpicture}
	\caption{A graph $H\in\super(G)$ with $\beta(H)=2$.}
	\label{k10b}
	\end{subfigure}
	\caption{A reducible graph for which the addition of any single edge does not reduce the dimension.}
\end{figure}

The following is an immediate consequence of Theorem~\ref{tree}.

\begin{corollary} \label{tree_dim_3_has_threshold2}
If $T$ is a tree with $\beta(T)=3$, then $\tau(T)=2$.  
\end{corollary}

Our next result takes this one step further; we show that if $\beta(T)=4$, then $\tau(T)=2$.
%We note that this is best possible in some sense, as one can show that $\beta(K_{1,6})=5$ and $\tau(K_{1,6})=3$.

\begin{theorem}\label{4tree}
	Let $T$ be a tree with $\beta(T)=4$. Then there exists a set $E=\{e_1,e_2\}$ of cardinality $2$ such that $\beta(T+E)=2$.
\end{theorem}

\begin{proof}
%The proof deals with different cases, and is largely similar to the proof of Theorem \ref{tree}. For this reason we only present an outline of the proof in this paper. A proof is included in an appendix.
Let $T$ be a tree with $\beta(T)=4$. We consider four cases, based on the number of vertices with terminal degree at least $2$ in $T$, which is at most four.

\smallskip
	
\noindent	\textbf{Case 1:} $T$ has exactly one vertex $v$ with terminal degree at least $2$.  This case is illustrated in Figure~\ref{Case1}.
\smallskip

\noindent In this case, we must have $\ter(v)=5$.  Let $B_1, ..., B_5$ be the five branches of $T$ at $v$, all of which must be paths. For every $i\in\{1,2,3,4,5\}$, let $v_i$ be the vertex of $B_i$ adjacent to $v$. Let $e_1=v_1v_2$, $e_2=v_3v_4$, and $F=T+\{e_1,e_2\}$.  We show that $\{v_1,v_4\}$ is a basis for $F$.

Let $x,y$ be any two distinct vertices of $F$.  Let $H_1=F-(V(B_2)-v)$, and let $H_2=F-(V(B_3)-v)$.  Note that $H_1$ and $H_2$ are isometric subgraphs of $F$.  From Case 1 of the proof of Theorem~\ref{tree}, we see that $\{v_1,v_4\}$ is a basis for both $H_1$ and $H_2$.  So if $x$ and $y$ both belong to $H_1$, or both belong to $H_2$, then they are resolved by the set $\{v_1,v_4\}$.  So we may assume, without loss of generality, that $x\in V(B_2)-v$ and $y\in V(B_3)-v$.  If $v_1$ resolves $x$ and $y$, then we are done, so suppose that $v_1$ does not resolve $x$ and $y$.  Let $k=d_F(v_1,x)=d_F(v_1,y)$.  Then we have $d_F(v_4,x)=k+1$ and $d_F(v_4,y)=k-1$, so we conclude that $v_4$ resolves $x$ and $y$.

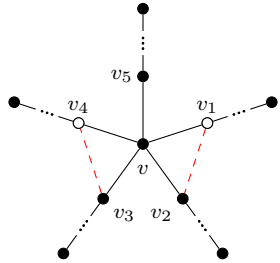
\begin{figure}[h]
		\centering
		\begin{tikzpicture}[scale=0.9]
		\vertex (0) at (0,0) {};
		\pgfmathtruncatemacro{\n}{5}		
		\foreach \x in {1,2,3,4,5}
		{
		\pgfmathtruncatemacro{\y}{\x+\n}
		\pgfmathtruncatemacro{\z}{18+72*\x}
		\vertex (\x) at (\z:1) {};
		\vertex (\y) at (\z:2) {};
		\path (\x) edge (0);
		\draw[middlearrow={ellipsis}{0.7}] (\x) -- (\y);
		}
		\node[below,yshift=-4pt] at (0) {\scriptsize $v$};
		\node[above] at (5) {\scriptsize $v_1$};
		\node[below left] at (4) {\scriptsize $v_2$};
		\node[below right] at (3) {\scriptsize $v_3$};
		\node[above] at (2) {\scriptsize $v_4$};
		\node[left] at (1) {\scriptsize $v_5$};
		\path[red,dashed]
		(5) edge (4)
		(3) edge (2);
		\whitevertex at (2) {};
		\whitevertex at (5) {};
		\end{tikzpicture}
	\caption{The tree $T$ in Case 1 of the proof of Theorem~\ref{4tree}.}
	\label{Case1}
\end{figure}

\smallskip
	
\noindent \textbf{Case 2:} $T$ has exactly two vertices $v_1$ and $v_2$ with terminal degree at least $2$.

\smallskip

\noindent
In this case, we have $\ter(v_1)+\ter(v_2)=6$.  For $i\in\{1,2\}$, let $B_{i1},B_{i2},\ldots, B_{i\ter(v_i)}$ be the branches of $T$ at $v_i$ that are paths.  For every $i\in\{1,2\}$ and $j\in\{1,\ldots,\ter(v_i)\}$, let $v_{ij}$ be the neighbour of $v_i$ that lies in the branch $B_{ij}$.  We deal with two subcases.

\smallskip
	
\noindent \textbf{Subcase 2.1:} $\ter(v_1)=4$ and $\ter(v_2)=2$. This case is illustrated in Figure~\ref{Subcase2.1}.

\smallskip

\noindent
Let $w$ be the neighbour of $v_1$ on the $v_1$--$v_2$ path in $T$.  Let $e_1=v_{11}v_{12}$, $e_2=v_{14}w$, and $F=T+\{e_1,e_2\}$. We show that  $\{v_{11},v_{21}\}$ is a basis for $F$.

Let $x,y$ be any two distinct vertices of $F$.  Let $H_1=F-(V(B_{12})-v_1)$ and $H_2=F-(V(B_{14})-v_1)$.  Note that $H_1$ and $H_2$ are isometric subgraphs of $F$.

Assume first that $x$ and $y$ belong to $H_1$. By Remark~\ref{TprimeRemark}, the argument used in Subcase 2.2 of Theorem~\ref{tree} establishes that the set $\{v_{11},v_{21}\}$ is a resolving set for $H_1$.
Assume next that $x$ and $y$ belong to $H_2$. From Case 1 of Theorem~\ref{tree}, we see that the set $\{v_{11},v_{21}\}$ is a resolving set for $H_2$.  So if $x$ and $y$ both belong to $H_1$, or both belong to $H_2$, then $x$ and $y$ are resolved by $\{v_{11},v_{21}\}$ in $F$.  So we may assume, without loss of generality, that $x$ belongs to $B_{12}-v_1$, and $y$ belongs to $B_{14}-v_1$.  Suppose that $v_{11}$ does not resolve $x$ and $y$. Let $k=d_F(v_{11},x)=d_F(v_{11},y)$ and $\ell=d_F(v_1,v_2)$.  Then we have $d_F(v_{21},x)=k+\ell+1$ and $d_F(v_{21},y)=k+\ell-1$, from which we conclude that $v_{21}$ resolves $x$ and $y$.  So if $v_{11}$ does not resolve $x$ and $y$, then $v_{21}$ does.

\begin{figure}[h]
		\centering
		\begin{subfigure}[t]{0.45\textwidth}
		\centering
		\begin{tikzpicture}[scale=0.9]
		\vertex (0) at (0,0) {};
		\pgfmathtruncatemacro{\n}{4}		
		\foreach \x in {1,2,3,4}
		{
		\pgfmathtruncatemacro{\y}{\x+\n}
		\pgfmathtruncatemacro{\z}{90+60*(\x-1)}
		\vertex (\x) at (\z:1) {};
		\vertex (\y) at (\z:2) {};
		\path (\x) edge (0);
		\draw[middlearrow={ellipsis}{0.7}] (\x) -- (\y);
		}
		\vertex (9) at (1,0) {};
		\vertex (10) at (2,0) {};
		\vertex (12) at (2.7,0.7) {};
		\vertex (13) at (3.4,1.4) {};
		\vertex (14) at (2.7,-0.7) {};
		\vertex (15) at (3.4,-1.4) {};
		\path
		(0) edge (9)
		(10) edge (12)
		(10) edge (14);
		\draw[middlearrow={ellipsis}{0.7}] (9) -- (10);
		\draw[middlearrow={ellipsis}{0.7}] (12) -- (13);
		\draw[middlearrow={ellipsis}{0.7}] (14) -- (15);
		\node[above right] at (0) {\scriptsize $v_1$};
		\node[right] at (1) {\scriptsize $v_{11}$};
		\node[below] at (2) {\scriptsize $v_{12}$};
		\node[above] at (3) {\scriptsize $v_{13}$};
		\node[left] at (4) {\scriptsize $v_{14}$};
		\node[above] at (9) {\scriptsize $w$};
		\node[above left] at (10) {\scriptsize $v_2$};
		\node[right] at (12) {\scriptsize $v_{21}$};
		\node[right] at (14) {\scriptsize $v_{22}$};
		\path[red,dashed]
		(1) edge (2)
		(4) edge (9);
		\whitevertex at (0,1) {};
		\whitevertex at (2.7,0.7) {};
		\end{tikzpicture}
		\caption{Subcase 2.1}
		\label{Subcase2.1}
		\end{subfigure}
		\begin{subfigure}[t]{0.45\textwidth}
		\centering
		\begin{tikzpicture}[scale=0.9]
		\vertex (0) at (0,0) {};
		\pgfmathtruncatemacro{\n}{3}		
		\foreach \x in {1,2,3}
		{
		\pgfmathtruncatemacro{\y}{\x+\n}
		\pgfmathtruncatemacro{\z}{90+90*(\x-1)}
		\vertex (\x) at (\z:1) {};
		\vertex (\y) at (\z:2) {};
		\path (\x) edge (0);
		\draw[middlearrow={ellipsis}{0.7}] (\x) -- (\y);
		}
		\vertex (10) at (2,0) {};
		\vertex (12) at (2,1) {};
		\vertex (13) at (2,2) {};
		\vertex (14) at (3,0) {};
		\vertex (15) at (4,0) {};
		\vertex (16) at (2,-1) {};
		\vertex (17) at (2,-2) {};
		\whitevertex at (0,1) {};
		\whitevertex at (2,1) {};		
		\path
		(10) edge (12)
		(10) edge (14)
		(10) edge (16);
		\draw[middlearrow={ellipsis}{0.6}] (0) -- (10);
		\draw[middlearrow={ellipsis}{0.7}] (12) -- (13);
		\draw[middlearrow={ellipsis}{0.7}] (14) -- (15);
		\draw[middlearrow={ellipsis}{0.7}] (16) -- (17);
		\node[above right] at (0) {\scriptsize $v_1$};
		\node[right] at (1) {\scriptsize $v_{11}$};
		\node[below] at (2) {\scriptsize $v_{12}$};
		\node[left] at (3) {\scriptsize $v_{13}$};
		\node[above left] at (10) {\scriptsize $v_2$};
		\node[left] at (12) {\scriptsize $v_{21}$};
		\node[below] at (14) {\scriptsize $v_{22}$};
		\node[right] at (16) {\scriptsize $v_{23}$};
		\path[red,dashed]
		(1) edge (2)
		(12) edge (14);
		\end{tikzpicture}
		\caption{Subcase 2.2}
		\label{Subcase2.2}
		\end{subfigure}
	\caption{The tree $T$ in Case 2 of the proof of Theorem~\ref{4tree}. Note that every internal vertex of the $v_1$--$v_2$ path in $T$ may be adjacent to a single limb -- these limbs are not drawn.}
	\label{Case2}
\end{figure}

\smallskip

\noindent \textbf{Subcase 2.2:} $\ter(v_1)=\ter(v_2)=3$. This case is illustrated in Figure~\ref{Subcase2.2}.

\smallskip

\noindent
Let $e_1=v_{11}v_{12}$, $e_2=v_{21}v_{22}$, and $F=T+\{e_1,e_2\}$. We show that $\{v_{11},v_{21}\}$ is a basis for $F$.

Let $x,y$ be any two distinct vertices of $F$.  Let $H_1=F-(V(B_{12})-v_1)$ and $H_2=F-(V(B_{22})-v_2)$.  Then $H_1$ and $H_2$ are isometric subgraphs of $F$, and from Case 1 of Theorem~\ref{tree}, the set $\{v_{11},v_{21}\}$ resolves both $H_1$ and $H_2$.  So we may assume, without loss of generality, that $x$ belongs to $B_{12}-v_1$, and $y$ belongs to $B_{22}-v_1$.  Suppose that $v_{11}$ does not resolve $x$ and $y$. Let $k=d_F(v_{11},x)=d_F(v_{11},y)$ and $\ell=d_F(v_1,v_2)$.  Then we have $d_F(v_{21},x)=k+\ell+1$ and $d_F(v_{21},y)=k-\ell-1$, from which we conclude that $v_{21}$ resolves $x$ and $y$.

\smallskip

\noindent	\textbf{Case 3:} $T$ has exactly three vertices $v_1$,  $v_2$, and $v_3$ with terminal degree at least $2$.

\smallskip

\noindent In this case, we must have $\ter(v_1)+\ter(v_2)+\ter(v_3)=7$.  For $i\in\{1,2,3\}$, let $B_{i1},B_{i2},\ldots, B_{i\ter(v_i)}$ be the branches of $T$ at $v_i$ that are paths.  For every $i\in\{1,2\}$ and $j\in\{1,\ldots,\ter(v_i)\}$, let $v_{ij}$ be the neighbour of $v_i$ that lies in the branch $B_{ij}$.

Let $S$ be the core of $T$. Since every leaf of $S$ must have terminal degree at least $2$ in $T$, we see that $S$ has at most three leaves.  We consider two subcases relative to $S$.

\smallskip

\noindent \textbf{Subcase 3.1:} $S$ has exactly three leaves.  This case is illustrated in Figure~\ref{Subcase3.1}.

\smallskip
	
\noindent Then exactly one vertex $v$ of $S$ has degree $3$, and every branch of $S$ at $v$ is a path. In this case, the leaves of $S$ must be $v_1$, $v_2$, and $v_3$.  For $i\in\{1,2,3\}$, let $B_i$ be the branch of $T$ at $v$ containing $v_i$, and let $d_i=d_T(v, v_{i})$. Without loss of generality, assume that $\ter(v_1)=\ter(v_2)=2$, that $\ter(v_3)=3$, and that $d_1\leq d_2$.  Let $w$ be the vertex on the $v$--$v_{21}$ subpath of $T$ that is distance $d_1+1$ from $v$ in $T$. Let $e_1=v_{31}v_{32}$, $e_2=v_{11}w$, and $F=T+\{e_1,e_2\}$. We show that $\{v_{21},v_{31}\}$ is a basis for $F$.

Let $x,y$ be any two distinct vertices of $F$.  Let $H_1=F-(V(B_{1})-v)$ and $H_2=F-(V(B_{32})-v_3)$.  Note that $H_1$ and $H_2$ are isometric subgraphs of $F$.  From Case 1 of Theorem~\ref{tree}, we see that $\{v_{21},v_{31}\}$ is a basis for $H_1$.  From Subcase 2.1 of Theorem~\ref{tree}, we see that $\{v_{21},v_{31}\}$ is a basis for $H_2$.  So we may assume, without loss of generality, that $x$ belongs to $B_{1}-v$, and $y$ belongs to $B_{32}-v_3$.  Let $x'$ be the vertex closest to $x$ on the $v$--$v_{11}$ path of $T$.
Suppose that $v_{21}$ does not resolve $x$ and $y$.  Let $k=d_F(v_{3},y)$.  Then $d_F(v_{31},y)=k$ as well. We also have
\[
1+d_2+d_3+k=d_F(v_{21},y)=d_F(v_{21},x)=1+d_2-d_1+d_F(v_{11},x')+d_F(x',x).
\]
It follows that
\[
d_F(x',x)=d_3+k+d_1-d_F(v_{11},x')\geq d_3+k,
\]
since $d_1\geq d_F(v_{11},x')$.  But then
\begin{align*}
d_F(v_{31},x)&=1+d_3+d_F(v,x')+d_F(x',x)\\
&\geq 1+2d_3+k+d_F(v,x')\\
&> 1+2d_3+k.
\end{align*}
So we have $d_F(v_{31},x)>k=d_F(v_{31},y)$, and we conclude that $v_{31}$ resolves $x$ and $y$.

\smallskip

\begin{figure}[htbp]
		\centering
		\begin{subfigure}[t]{0.9\textwidth}
		\centering
		\begin{tikzpicture}[scale=0.9]
		\vertex (0) at (0,0) {};
		\vertex (-1) at (-1,0) {};
		\begin{scope}[shift={(-1,0)}]
		\foreach \r in {1,2}
		{
		\foreach \t in {1,2,3}
		{
		\vertex (a\r\t) at (45*\t+90:\r) {};
		}
		}
		\end{scope}
		\foreach \r in {1,3,4,5,6}
		{
		\vertex (b\r) at (30:\r) {};
		}
		\foreach \r in {1,2,3,4}
		{
		\vertex (c\r) at (-30:\r) {};
		}
		\begin{scope}[shift={(b4)}]
		\foreach \r in {1,2}
		{
		\vertex (d\r) at (-30:\r) {};
		}
		\end{scope}
		\begin{scope}[shift={(c2)}]
		\foreach \r in {1,2}
		{
		\vertex (e\r) at (30:\r) {};
		}
		\end{scope}
		\path
		(0) edge (b1)
		(0) edge (c1)
		(-1) edge (a11)
		(-1) edge (a12)
		(-1) edge (a13)
		(b4) edge (b5)
		(b4) edge (d1)
		(c2) edge (c3)
		(c2) edge (e1)
		;
		\draw[middlearrow={ellipsis}{0.7}] (a11) -- (a21);
		\draw[middlearrow={ellipsis}{0.7}] (a12) -- (a22);
		\draw[middlearrow={ellipsis}{0.7}] (a13) -- (a23);
		\draw[middlearrow={ellipsis}{0.6}] (b1) -- (b3);
		\draw[middlearrow={ellipsis}{0.7}] (b3) -- (b4);
		\draw[middlearrow={ellipsis}{0.7}] (c1) -- (c2);
		\draw[middlearrow={ellipsis}{0.7}] (b5) -- (b6);
		\draw[middlearrow={ellipsis}{0.7}] (c3) -- (c4);
		\draw[middlearrow={ellipsis}{0.7}] (d1) -- (d2);
		\draw[middlearrow={ellipsis}{0.7}] (e1) -- (e2);
		\draw[middlearrow={ellipsis}{0.7}] (-1) -- (0);
		\node[below,yshift=-3pt] at (0) {\scriptsize $v$};
		\node[below,yshift=-3pt] at (-1) {\scriptsize $v_3$};
		\node[below right] at (a13) {\scriptsize $v_{31}$};
		\node[above] at (a12) {\scriptsize $v_{32}$};
		\node[above right] at (a11) {\scriptsize $v_{33}$};
		\node[above] at (b4) {\scriptsize $v_2$};
		\node[below] at (c2) {\scriptsize $v_1$};
		\node[above] at (b5) {\scriptsize $v_{22}$};
		\node[below] at (d1) {\scriptsize $v_{21}$};
		\node[below] at (e1) {\scriptsize $v_{11}$};
		\node[below] at (c3) {\scriptsize $v_{12}$};
		\node[above] at (b3) {\scriptsize $w$};
		\path[dashed,red]
		(a13) edge (a12)
		(e1) edge (b3);
		\whitevertex at (a13) {};
		\whitevertex at (d1) {};	
		\end{tikzpicture}
		\caption{Subcase 3.1}
		\label{Subcase3.1}
		\end{subfigure}\\\smallskip
		\begin{subfigure}[t]{0.9\textwidth}
		\centering
		\begin{tikzpicture}[scale=0.9]
		\foreach \x in {-5,-4,-3,-1,0,1,3}
		{
		\vertex (\x) at (\x,0) {};
		}
		\foreach \y in {-1,-2}
		{
		\vertex (a\y) at (-3+\y*0.7,\y*0.7) {};
		\vertex (b\y) at (3-\y*0.7,\y*0.7) {};
		\vertex (c\y) at (225:\y) {};
		}
		\foreach \y in {1,2}
		{
		\vertex (a\y) at (-3-\y*0.7,\y*0.7) {};
		\vertex (b\y) at (3+\y*0.7,\y*0.7) {};
		\vertex (c\y) at (135:\y) {};
		}
		\path
		(-4) edge (-3)
		(-1) edge (0)
		(0) edge (1)
		(-3) edge (a1)
		(-3) edge (a-1)
		(0) edge (c1)
		(0) edge (c-1)
		(3) edge (b1)
		(3) edge (b-1)
		;
		\draw[middlearrow={ellipsis}{0.7}] (a1) -- (a2);
		\draw[middlearrow={ellipsis}{0.7}] (a-1) -- (a-2);
		\draw[middlearrow={ellipsis}{0.7}] (b1) -- (b2);
		\draw[middlearrow={ellipsis}{0.7}] (b-1) -- (b-2);
		\draw[middlearrow={ellipsis}{0.7}] (c1) -- (c2);
		\draw[middlearrow={ellipsis}{0.7}] (c-1) -- (c-2);
		\draw[middlearrow={ellipsis}{0.6}] (-1) -- (-3);
		\draw[middlearrow={ellipsis}{0.6}] (1) -- (3);
		\draw[middlearrow={ellipsis}{0.7}] (-4) -- (-5);
		\node[above right] at (-3) {\scriptsize $v_1$};
		\node[above left] at (3) {\scriptsize $v_2$};
		\node[below] at (0) {\scriptsize $v_3$};
		\node[above right] at (a1) {\scriptsize $v_{11}$};
		\node[below] at (-4) {\scriptsize $v_{12}$};
		\node[below right] at (a-1) {\scriptsize $v_{13}$};
		\node[above left] at (b1) {\scriptsize $v_{21}$};
		\node[below left] at (b-1) {\scriptsize $v_{22}$};
		\node[right] at (c-1) {\scriptsize $v_{32}$};
		\node[left] at (c1) {\scriptsize $v_{31}$};
		\node[below] at (-1) {\scriptsize $w$};
		\path[red,dashed]
		(a1) edge (-4)
		(-1) edge (c1);
		\whitevertex at (a1) {};
		\whitevertex at (b1) {};
		\end{tikzpicture}
		\caption{Subcase 3.2.1}
		\label{Subcase3.2.1}
		\end{subfigure}\\\smallskip
		\begin{subfigure}[t]{0.9\textwidth}
		\centering
		\begin{tikzpicture}[scale=0.9]
		\vertex (0) at (0,0) {};
		\foreach \r in {1,2}
		{
		\foreach \t in {1,2,3}
		{
		\vertex (a\r\t) at (45*\t+90:\r) {};
		}
		}
		\foreach \r in {1,3,4,5,6}
		{
		\vertex (b\r) at (30:\r) {};
		}
		\foreach \r in {1,2,3,4}
		{
		\vertex (c\r) at (-30:\r) {};
		}
		\begin{scope}[shift={(b4)}]
		\foreach \r in {1,2}
		{
		\vertex (d\r) at (-30:\r) {};
		}
		\end{scope}
		\begin{scope}[shift={(c2)}]
		\foreach \r in {1,2}
		{
		\vertex (e\r) at (30:\r) {};
		}
		\end{scope}
		\path
		(0) edge (b1)
		(0) edge (c1)
		(0) edge (a11)
		(0) edge (a12)
		(0) edge (a13)
		(b4) edge (b5)
		(b4) edge (d1)
		(c2) edge (c3)
		(c2) edge (e1)
		;
		\draw[middlearrow={ellipsis}{0.7}] (a11) -- (a21);
		\draw[middlearrow={ellipsis}{0.7}] (a12) -- (a22);
		\draw[middlearrow={ellipsis}{0.7}] (a13) -- (a23);
		\draw[middlearrow={ellipsis}{0.6}] (b1) -- (b3);
		\draw[middlearrow={ellipsis}{0.7}] (b3) -- (b4);
		\draw[middlearrow={ellipsis}{0.7}] (c1) -- (c2);
		\draw[middlearrow={ellipsis}{0.7}] (b5) -- (b6);
		\draw[middlearrow={ellipsis}{0.7}] (c3) -- (c4);
		\draw[middlearrow={ellipsis}{0.7}] (d1) -- (d2);
		\draw[middlearrow={ellipsis}{0.7}] (e1) -- (e2);
		\node[below,yshift=-3pt] at (0) {\scriptsize $v_3$};
		\node[below right] at (a13) {\scriptsize $v_{31}$};
		\node[above] at (a12) {\scriptsize $v_{32}$};
		\node[above right] at (a11) {\scriptsize $v_{33}$};
		\node[above] at (b4) {\scriptsize $v_2$};
		\node[below] at (c2) {\scriptsize $v_1$};
		\node[above] at (b5) {\scriptsize $v_{22}$};
		\node[below] at (d1) {\scriptsize $v_{21}$};
		\node[below] at (e1) {\scriptsize $v_{11}$};
		\node[below] at (c3) {\scriptsize $v_{12}$};
		\node[above] at (b3) {\scriptsize $w$};
		\path[dashed,red]
		(a13) edge (a12)
		(e1) edge (b3);
		\whitevertex at (a13) {};
		\whitevertex at (d1) {};
		\end{tikzpicture}
		\caption{Subcase 3.2.2}
		\label{Subcase3.2.2}
		\end{subfigure}
	\caption{The tree $T$ in Case 3 of the proof of Theorem~\ref{4tree}. Note that for every $i,j\in\{1,2,3\}$, every internal vertex of the $v_i$--$v_j$ path in $T$ may be adjacent to a single limb -- these limbs are not drawn.}
	\label{Case3}
\end{figure}

\noindent \textbf{Subcase 3.2:} $S$ has exactly two leaves, i.e., $S$ is a path.

\smallskip

\noindent
In this case, both leaves of $S$ must belong to the set $\{v_1,v_2,v_3\}$. Without loss of generality, assume that $v_1$ and $v_2$ are the leaves of $S$, and that $\ter(v_1) \ge \ter(v_2)$. The vertex $v_3$ must lie on the $v_1$--$v_2$ path in $T$. We consider two further subcases.

\smallskip
	
\noindent \textbf{Subcase 3.2.1:} $\ter(v_3)=2$.  This case is illustrated in Figure~\ref{Subcase3.2.1}.

\smallskip
	
\noindent Let $w$ be the neighbour of $v_3$ on the $v_1$--$v_3$ path in $T$.  Let $e_1=v_{11}v_{12}$, $e_2=v_{31}w$, and $F=T+\{e_1, e_2\}$. We show that $\{v_{11}, v_{21}\}$ resolves $F$.

Let $x,y$ be any two distinct vertices of $F$. Let $H_1=F-V(B_{12}-v_1)$ and $H_2=F-(V(B_{31})-v_3)$.  Note that $H_1$ and $H_2$ are isometric subgraphs of $F$.  From Subcase 2.1 of Theorem~\ref{tree}, we see that $\{v_{11},v_{21}\}$ is a basis for $H_1$.  From Case 1 of Theorem~\ref{tree}, we see that $\{v_{11},v_{21}\}$ is a basis for $H_2$.  So we may assume, without loss of generality, that $x$ belongs to $B_{12}-v_1$, and $y$ belongs to $B_{31}-v_3$.  If $v_{11}$ does not resolve $x$ and $y$, then let $k=d_F(v_{11},x)=d_F(v_{11},y)$.  Then $d_F(v_{21},x)=d(v_{21},v_1)+k$, while $d_F(v_{21},y)=d_F(v_{21},v_3)+k-d(v_1,v_3)<d_F(v_{21},x)$, and hence $x$ and $y$ are resolved by $v_{21}$.

\smallskip

\noindent \textbf{Subcase 3.2.2:} $\ter(v_3)=3$. This case is illustrated in Figure~\ref{Subcase3.2.2}.

\smallskip
	
\noindent  We may assume that $d_1=d_T(v_{1},v_3) \le d_T(v_{2},v_3)$. Let $w$ be the vertex on the $v_3$--$v_{21}$ path that is distance $d_1+1$ from $v_3$. Let $e_1=v_{11}w$, $e_2=v_{31}v_{32}$, and $F= T+\{e_1,e_2\}$. We claim that  $\{v_{21},v_{31}\}$ is a basis for $F$.  The proof is the same as that of Subcase 3.1, with $v$ set equal to $v_3$.  (Note that the proof still works with $d_3=d(v,v_3)=0$.)

\smallskip
	
\noindent	\textbf{Case 4:} $T$ has exactly four vertices $v_1$, $v_2$, $v_3$, and $v_4$ with terminal degree at least~$2$.

\smallskip

\noindent
In this case, we must have $\ter(v_i)=2$ for all $i\in\{1,2,3,4\}$.  For $i\in\{1,2,3,4\}$, let $B_{i1}$ and $B_{i2}$ be the branches of $T$ at $v_i$ that are paths, and for $j\in\{1,2\}$, let $v_{ij}$ be the neighbour of $v_i$ in $B_{ij}$. Let $S$ be the core of $T$. Since every leaf of $S$ must be an exterior major vertex with terminal degree at least $2$, we see that $S$ can have at most four leaves.

\smallskip
	
\noindent \textbf{Subcase 4.1:} $S$ has exactly two leaves.  This case is illustrated in Figure~\ref{Subcase4.1}.
	
\noindent Without loss of generality, let $v_1, v_2, v_3, v_4$ appear in that order on $S$. So $v_1$ and $v_4$ are the leaves of $S$. Let $w$ be the neighbour of $v_2$ on the $v_1$--$v_2$ path of $T$, and let $z$ be the neighbour of $v_3$ on the $v_3$--$v_4$ path. Let $e_1=v_{21}w$, $e_2=v_{31}z$ and $F=T+\{e_1, e_2\}$. We show that $\{v_{11}, v_{41}\}$ is a basis for $F$.

Let $x,y$ be any two distinct vertices of $F$. Let $H_1=F-(V(B_{21})-v_2)$ and $H_2=F-(V(B_{31})-v_3)$. Note that both $H_1$ and $H_2$ are isometric subgraphs of $F$.  From Subcase 2.2 of Theorem \ref{tree}, we see that $\{v_{11}, v_{41}\}$ is a resolving set for both $H_1$ and $H_2$. So we may assume, without loss of generality, that $x$ belongs to $B_{21}-v_2$, and $y$ belongs to $B_{31}-v_3$.  Let $d_{12}=d_F(v_1,v_2)$, $d_{23}=d_F(v_2,v_3)$, and $d_{34}=d_F(v_3,v_4)$. Let $k=d_F(v_{21}, x)$ and $\ell=d_F(v_{31}, y)$. Suppose that $v_{11}$ does not resolve $x$ and $y$. Then $1+d_{12}+k=1+d_{12}+d_{23}+1+\ell$, and hence $k=d_{23}+\ell +1$. Since $d_F(v_{41}, y)=d_{34} + \ell +1$, we have
\begin{align*}
d_F(v_{41}, x) &= 1+d_{34}+d_{23}+k+1\\
&=2+d_{34}+d_{23}+(d_{23}+ \ell +1)\\
&=(2+2d_{23}) +(d_{34}+ \ell +1)\\
&> d_F(v_{41}, y).
\end{align*}
So if $v_{11}$ does not resolve $x$ and $y$, then $v_{41}$ does.

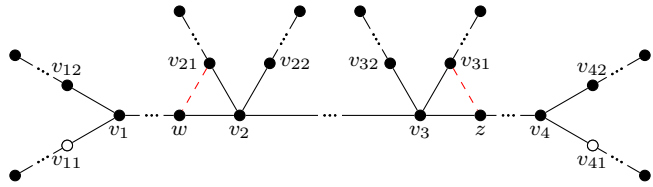
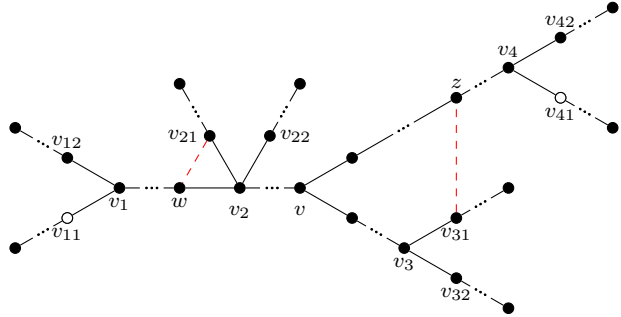
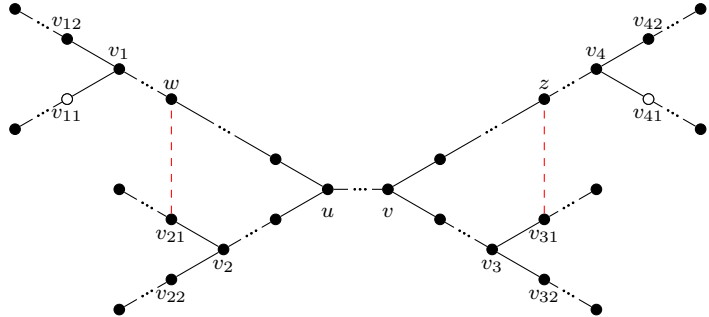
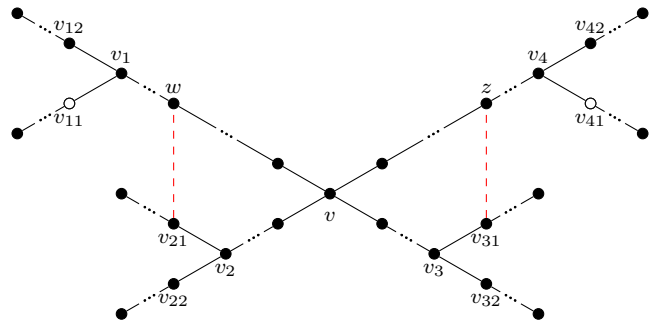
\begin{figure}[htbp]
		\centering
		\begin{subfigure}[t]{0.9\textwidth}
		\centering
		\begin{tikzpicture}[scale=0.8]
		\foreach \x in {-3,-2,-1}
		{
		\vertex (\x) at (\x-0.5,0) {};
		}
		\foreach \x in {1,2,3}
		{
		\vertex (\x) at (\x+0.5,0) {};
		}
		\begin{scope}[shift=(-3)]
		\vertex (a1) at (150:1) {};
		\vertex (a2) at (150:2) {};
		\vertex (a3) at (210:1) {};
		\vertex (a4) at (210:2) {};
		\end{scope}
		\begin{scope}[shift=(3)]
		\vertex (b1) at (-30:1) {};
		\vertex (b2) at (-30:2) {};
		\vertex (b3) at (30:1) {};
		\vertex (b4) at (30:2) {};
		\end{scope}
		\begin{scope}[shift=(-1)]
		\vertex (c1) at (60:1) {};
		\vertex (c2) at (60:2) {};
		\vertex (c3) at (120:1) {};
		\vertex (c4) at (120:2) {};
		\end{scope}
		\begin{scope}[shift=(1)]
		\vertex (d1) at (60:1) {};
		\vertex (d2) at (60:2) {};
		\vertex (d3) at (120:1) {};
		\vertex (d4) at (120:2) {};
		\end{scope}
		\path
		(-2) edge (-1)
		(1) edge (2)
		(-3) edge (a1)
		(-3) edge (a3)
		(3) edge (b1)
		(3) edge (b3)
		(-1) edge (c1)
		(-1) edge (c3)
		(1) edge (d1)
		(1) edge (d3)
		;
		\draw[middlearrow={ellipsis}{0.7}] (a1) -- (a2);
		\draw[middlearrow={ellipsis}{0.7}] (a3) -- (a4);
		\draw[middlearrow={ellipsis}{0.7}] (b1) -- (b2);
		\draw[middlearrow={ellipsis}{0.7}] (b3) -- (b4);
		\draw[middlearrow={ellipsis}{0.7}] (c1) -- (c2);
		\draw[middlearrow={ellipsis}{0.7}] (c3) -- (c4);
		\draw[middlearrow={ellipsis}{0.7}] (d1) -- (d2);
		\draw[middlearrow={ellipsis}{0.7}] (d3) -- (d4);
		\draw[middlearrow={ellipsis}{0.7}] (-2) -- (-3);
		\draw[middlearrow={ellipsis}{0.7}] (2) -- (3);
		\draw[middlearrow={ellipsis}{0.57}] (-1) -- (1);
		\node[below] at (-3) {\scriptsize $v_1$};
		\node[below] at (-1) {\scriptsize $v_2$};
		\node[below] at (1) {\scriptsize $v_3$};
		\node[below] at (3) {\scriptsize $v_4$};
		\node[below] at (a3) {\scriptsize $v_{11}$};
		\node[above] at (a1) {\scriptsize $v_{12}$};
		\node[below] at (-2) {\scriptsize $w$};
		\node[left] at (c3) {\scriptsize $v_{21}$};
		\node[right] at (c1) {\scriptsize $v_{22}$};
		\node[below] at (2) {\scriptsize $z$};
		\node[left] at (d3) {\scriptsize $v_{32}$};
		\node[right] at (d1) {\scriptsize $v_{31}$};
		\node[above] at (b3) {\scriptsize $v_{42}$};
		\node[below] at (b1) {\scriptsize $v_{41}$};
		\path[red,dashed]
		(c3) edge (-2)
		(d1) edge (2);
		\whitevertex at (a3) {};
		\whitevertex at (b1) {};
		\end{tikzpicture}
		\caption{Subcase 4.1}
		\label{Subcase4.1}
		\end{subfigure}\\\smallskip
		\begin{subfigure}[t]{0.9\textwidth}
		\centering
		\begin{tikzpicture}[scale=0.8]
		\vertex (0) at (0,0) {};
		\vertex (-1) at (-1,0) {};
		\vertex (-2) at (-2,0) {};
		\vertex (-3) at (-3,0) {};
		\begin{scope}[shift={(-3,0)}]
		\foreach \r in {1,2}
		{
		\foreach \t in {1,2}
		{
		\vertex (a\r\t) at (60*\t+90:\r) {};
		}
		}
		\end{scope}
		\foreach \r in {1,3,4,5,6}
		{
		\vertex (b\r) at (30:\r) {};
		}
		\foreach \r in {1,2,3,4}
		{
		\vertex (c\r) at (-30:\r) {};
		}
		\begin{scope}[shift={(b4)}]
		\foreach \r in {1,2}
		{
		\vertex (d\r) at (-30:\r) {};
		}
		\end{scope}
		\begin{scope}[shift={(c2)}]
		\foreach \r in {1,2}
		{
		\vertex (e\r) at (30:\r) {};
		}
		\end{scope}
		\begin{scope}[shift=(-1)]
		\foreach \r in {1,2}
		{
		\foreach \t in {1,2}
		{
		\vertex (f\r\t) at (60*\t:\r) {};
		}
		}
		\end{scope}
		\path
		(0) edge (b1)
		(0) edge (c1)
		(-3) edge (a11)
		(-3) edge (a12)
		(b4) edge (b5)
		(b4) edge (d1)
		(c2) edge (c3)
		(c2) edge (e1)
		(-2) edge (-1)
		(-1) edge (f11)
		(-1) edge (f12)
		;
		\draw[middlearrow={ellipsis}{0.7}] (a11) -- (a21);
		\draw[middlearrow={ellipsis}{0.7}] (a12) -- (a22);
		\draw[middlearrow={ellipsis}{0.7}] (f11) -- (f21);
		\draw[middlearrow={ellipsis}{0.7}] (f12) -- (f22);
		\draw[middlearrow={ellipsis}{0.6}] (b1) -- (b3);
		\draw[middlearrow={ellipsis}{0.7}] (b3) -- (b4);
		\draw[middlearrow={ellipsis}{0.7}] (c1) -- (c2);
		\draw[middlearrow={ellipsis}{0.7}] (b5) -- (b6);
		\draw[middlearrow={ellipsis}{0.7}] (c3) -- (c4);
		\draw[middlearrow={ellipsis}{0.7}] (d1) -- (d2);
		\draw[middlearrow={ellipsis}{0.7}] (e1) -- (e2);
		\draw[middlearrow={ellipsis}{0.7}] (-2) -- (-3);
		\draw[middlearrow={ellipsis}{0.7}] (0) -- (-1);
		\node[below,yshift=-3pt] at (0) {\scriptsize $v$};
		\node[below,yshift=-3pt] at (-1) {\scriptsize $v_2$};
		\node[below] at (-3) {\scriptsize $v_{1}$};
		\node[below] at (a12) {\scriptsize $v_{11}$};
		\node[above] at (a11) {\scriptsize $v_{12}$};
		\node[left] at (f12) {\scriptsize $v_{21}$};
		\node[right] at (f11) {\scriptsize $v_{22}$};
		\node[below] at (-2) {\scriptsize $w$};
		\node[above] at (b4) {\scriptsize $v_4$};
		\node[below] at (c2) {\scriptsize $v_3$};
		\node[above] at (b5) {\scriptsize $v_{42}$};
		\node[below] at (d1) {\scriptsize $v_{41}$};
		\node[below] at (e1) {\scriptsize $v_{31}$};
		\node[below] at (c3) {\scriptsize $v_{32}$};
		\node[above] at (b3) {\scriptsize $z$};
		\path[dashed,red]
		(f12) edge (-2)
		(e1) edge (b3);
		\whitevertex at (a12) {};
		\whitevertex at (d1) {};	
		\end{tikzpicture}
		\caption{Subcase 4.2}
		\label{Subcase4.2}
		\end{subfigure}\\\smallskip
		\begin{subfigure}[t]{0.9\textwidth}
		\centering
		\begin{tikzpicture}[scale=0.8]
		\vertex (0) at (0,0) {};
		\vertex (-1) at (-1,0) {};
		\foreach \r in {1,3,4,5,6}
		{
		\vertex (b\r) at (30:\r) {};
		}
		\foreach \r in {1,2,3,4}
		{
		\vertex (c\r) at (-30:\r) {};
		}
		\begin{scope}[shift={(b4)}]
		\foreach \r in {1,2}
		{
		\vertex (d\r) at (-30:\r) {};
		}
		\end{scope}
		\begin{scope}[shift={(c2)}]
		\foreach \r in {1,2}
		{
		\vertex (e\r) at (30:\r) {};
		}
		\end{scope}
		\begin{scope}[shift=(-1)]
		\foreach \r in {1,3,4,5,6}
		{
		\vertex (f\r) at (150:\r) {};
		}
		\foreach \r in {1,2,3,4}
		{
		\vertex (g\r) at (210:\r) {};
		}
		\end{scope}
		\begin{scope}[shift=(f4)]
		\foreach \r in {1,2}
		{
		\vertex (h\r) at (210:\r) {};
		}
		\end{scope}
		\begin{scope}[shift=(g2)]
		\foreach \r in {1,2}
		{
		\vertex (i\r) at (150:\r) {};
		}
		\end{scope}
		\path
		(0) edge (b1)
		(0) edge (c1)
		(b4) edge (b5)
		(b4) edge (d1)
		(c2) edge (c3)
		(c2) edge (e1)
		(-1) edge (f1)
		(-1) edge (g1)
		(f4) edge (f5)
		(f4) edge (h1)
		(g2) edge (g3)
		(g2) edge (i1)
		;
		\draw[middlearrow={ellipsis}{0.6}] (b1) -- (b3);
		\draw[middlearrow={ellipsis}{0.7}] (b3) -- (b4);
		\draw[middlearrow={ellipsis}{0.7}] (c1) -- (c2);
		\draw[middlearrow={ellipsis}{0.7}] (b5) -- (b6);
		\draw[middlearrow={ellipsis}{0.7}] (c3) -- (c4);
		\draw[middlearrow={ellipsis}{0.7}] (d1) -- (d2);
		\draw[middlearrow={ellipsis}{0.7}] (e1) -- (e2);
		\draw[middlearrow={ellipsis}{0.7}] (0) -- (-1);
		\draw[middlearrow={ellipsis}{0.6}] (f1) -- (f3);
		\draw[middlearrow={ellipsis}{0.7}] (g1) -- (g2);
		\draw[middlearrow={ellipsis}{0.7}] (f3) -- (f4);
		\draw[middlearrow={ellipsis}{0.7}] (f5) -- (f6);
		\draw[middlearrow={ellipsis}{0.7}] (h1) -- (h2);
		\draw[middlearrow={ellipsis}{0.7}] (g3) -- (g4);
		\draw[middlearrow={ellipsis}{0.7}] (i1) -- (i2);
		\node[below,yshift=-3pt] at (0) {\scriptsize $v$};
		\node[below,yshift=-3pt] at (-1) {\scriptsize $u$};
		\node[above] at (b4) {\scriptsize $v_4$};
		\node[below] at (c2) {\scriptsize $v_3$};
		\node[above] at (b5) {\scriptsize $v_{42}$};
		\node[below] at (d1) {\scriptsize $v_{41}$};
		\node[below] at (e1) {\scriptsize $v_{31}$};
		\node[below] at (c3) {\scriptsize $v_{32}$};
		\node[above] at (b3) {\scriptsize $z$};
		\node[above] at (f4) {\scriptsize $v_1$};
		\node[below] at (g2) {\scriptsize $v_2$};
		\node[above] at (f5) {\scriptsize $v_{12}$};
		\node[below] at (h1) {\scriptsize $v_{11}$};
		\node[below] at (i1) {\scriptsize $v_{21}$};
		\node[below] at (g3) {\scriptsize $v_{22}$};
		\node[above] at (f3) {\scriptsize $w$};
		\path[dashed,red]
		(e1) edge (b3)
		(i1) edge (f3);
		\whitevertex at (d1) {};
		\whitevertex at (h1) {};
		\end{tikzpicture}
		\caption{Subcase 4.3.1}
		\label{Subcase4.3.1}
		\end{subfigure}\\\smallskip
				\begin{subfigure}[t]{0.9\textwidth}
		\centering
		\begin{tikzpicture}[scale=0.8]
		\vertex (0) at (0,0) {};
		\foreach \r in {1,3,4,5,6}
		{
		\vertex (b\r) at (30:\r) {};
		}
		\foreach \r in {1,2,3,4}
		{
		\vertex (c\r) at (-30:\r) {};
		}
		\begin{scope}[shift={(b4)}]
		\foreach \r in {1,2}
		{
		\vertex (d\r) at (-30:\r) {};
		}
		\end{scope}
		\begin{scope}[shift={(c2)}]
		\foreach \r in {1,2}
		{
		\vertex (e\r) at (30:\r) {};
		}
		\end{scope}
		\foreach \r in {1,3,4,5,6}
		{
		\vertex (f\r) at (150:\r) {};
		}
		\foreach \r in {1,2,3,4}
		{
		\vertex (g\r) at (210:\r) {};
		}
		\begin{scope}[shift=(f4)]
		\foreach \r in {1,2}
		{
		\vertex (h\r) at (210:\r) {};
		}
		\end{scope}
		\begin{scope}[shift=(g2)]
		\foreach \r in {1,2}
		{
		\vertex (i\r) at (150:\r) {};
		}
		\end{scope}
		\path
		(0) edge (b1)
		(0) edge (c1)
		(b4) edge (b5)
		(b4) edge (d1)
		(c2) edge (c3)
		(c2) edge (e1)
		(0) edge (f1)
		(0) edge (g1)
		(f4) edge (f5)
		(f4) edge (h1)
		(g2) edge (g3)
		(g2) edge (i1)
		;
		\draw[middlearrow={ellipsis}{0.6}] (b1) -- (b3);
		\draw[middlearrow={ellipsis}{0.7}] (b3) -- (b4);
		\draw[middlearrow={ellipsis}{0.7}] (c1) -- (c2);
		\draw[middlearrow={ellipsis}{0.7}] (b5) -- (b6);
		\draw[middlearrow={ellipsis}{0.7}] (c3) -- (c4);
		\draw[middlearrow={ellipsis}{0.7}] (d1) -- (d2);
		\draw[middlearrow={ellipsis}{0.7}] (e1) -- (e2);
		\draw[middlearrow={ellipsis}{0.6}] (f1) -- (f3);
		\draw[middlearrow={ellipsis}{0.7}] (g1) -- (g2);
		\draw[middlearrow={ellipsis}{0.7}] (f3) -- (f4);
		\draw[middlearrow={ellipsis}{0.7}] (f5) -- (f6);
		\draw[middlearrow={ellipsis}{0.7}] (h1) -- (h2);
		\draw[middlearrow={ellipsis}{0.7}] (g3) -- (g4);
		\draw[middlearrow={ellipsis}{0.7}] (i1) -- (i2);
		\node[below,yshift=-3pt] at (0) {\scriptsize $v$};
		\node[above] at (b4) {\scriptsize $v_4$};
		\node[below] at (c2) {\scriptsize $v_3$};
		\node[above] at (b5) {\scriptsize $v_{42}$};
		\node[below] at (d1) {\scriptsize $v_{41}$};
		\node[below] at (e1) {\scriptsize $v_{31}$};
		\node[below] at (c3) {\scriptsize $v_{32}$};
		\node[above] at (b3) {\scriptsize $z$};
		\node[above] at (f4) {\scriptsize $v_1$};
		\node[below] at (g2) {\scriptsize $v_2$};
		\node[above] at (f5) {\scriptsize $v_{12}$};
		\node[below] at (h1) {\scriptsize $v_{11}$};
		\node[below] at (i1) {\scriptsize $v_{21}$};
		\node[below] at (g3) {\scriptsize $v_{22}$};
		\node[above] at (f3) {\scriptsize $w$};
		\path[dashed,red]
		(e1) edge (b3)
		(i1) edge (f3);
		\whitevertex at (d1) {};
		\whitevertex at (h1) {};
		\end{tikzpicture}
		\caption{Subcase 4.3.2}
		\label{Subcase4.3.2}
		\end{subfigure}
	\caption{The tree $T$ in Case 4 of the proof of Theorem~\ref{4tree}. Note that for every $i,j\in\{1,2,3,4\}$, every internal vertex of the $v_i$--$v_j$ path in $T$ may be adjacent to a single limb -- these limbs are not drawn.}
	\label{Case4}
\end{figure}

\smallskip
	
\noindent \textbf{Subcase 4.2:} $S$ has exactly three leaves.  This case is illustrated in Figure~\ref{Subcase4.2}.

\smallskip

\noindent Then there is a single vertex $v$ of degree $3$ in $S$.  Without loss of generality, suppose that $v_1$, $v_3$, and $v_4$ are leaves of $S$, and that $v_2$ is an internal vertex of $S$.  Without loss of generality, assume that $v_2$ is on the the $v$--$v_1$ path of $T$.  Note that it is possible that $v_2=v$.
	
For $i\in\{1,2,3,4\}$, let $d_{i}=d_T(v, v_i)$.  Without loss of generality, assume that $d_{3}\leq d_{4}$.  Let $w$  be the neighbour of $v_2$ on the $v_1$--$v_2$ subpath of $T$, and let $z$ be the vertex distance $d_{3}+1$ from $v$ on the $v$-$v_{4}$ subpath of $T$. Let $e_1=v_{21}w$, $e_2= v_{31}z$, and  $F=T+\{e_1,e_2\}$. We show that $\{v_{11}, v_{41}\}$ is a resolving set for $F$.

Let $x,y$ be any two distinct vertices of $F$.  Let $B_3$ be the branch of $T$ at $v$ containing $v_3$.  Let $H_1=F-(V(B_3)-v)$ and $H_2=F-(V(B_{21})-v_2)$.  Note that $H_1$ and $H_2$ are isometric subgraphs of $F$.  From Subcase 2.2 of Theorem~\ref{tree}, we see that $\{v_{11},v_{41}\}$ is a basis for $H_1$.  From Subcase 2.1 of Theorem~\ref{tree}, we see that $\{v_{11},v_{41}\}$ is a basis for $H_2$.

So we may assume, without loss of generality, that $x$ belongs to $B_3-v$, and $y$ belongs to $B_{21}-v_2$.  Let $x'$ be the vertex on the $v$--$v_{31}$ path of $T$ that is closest to $x$ in $T$.  Let $k=d_T(v_{21},y)$.  Then
\begin{align*}
d_F(v_{11},x)&=1+d_1+d_F(v,x')+d_F(x',x),\\
d_F(v_{41},x)&=1+d_4-d_F(v,x')+1+d_F(x',x),\\
d_F(v_{11},y)&=1+d_1-d_2+k, \text{ and}\\
d_F(v_{41},y)&=1+d_4+d_2+1+k.
\end{align*}
Suppose that $x$ and $y$ are not resolved by $\{v_{11},v_{41}\}$.  Then we have $d_F(v_{11},x)=d_F(v_{11},y)$ and $d_F(v_{41},x)=d_F(v_{41},y)$.  Thus we have
\[
d_F(v_{11},x)+d_F(v_{41},x)=d_F(v_{11},y)+d_F(v_{41},y),
\]
from which it follows that $k=d_F(x,x')$.  However, then we have
\[
d_F(v_{11},x)=1+d_1+d_F(v,x')+k>1+d_1+k\geq d_F(v_{11},y),
\]
which is a contradiction.

\smallskip
	
\noindent \textbf{Subcase 4.3:} $S$ has exactly four leaves.

\smallskip

\noindent
In this case, the four leaves of $S$ must be the exterior major vertices $v_1$, $v_2$, $v_3$, and~$v_4$.  We consider two further subcases.

\smallskip
	
\noindent \textbf{Subcase 4.3.1:} $S$ contains exactly two major vertices $u$ and $v$.  This case is illustrated in Figure~\ref{Subcase4.3.1}.

\smallskip
	
\noindent In this case, both $u$ and $v$ must have degree $3$ in $S$. Without loss of generality, we may assume that $v_1$ and $v_2$ lie on distinct limbs of $u$ in $S$, while $v_3$ and $v_4$ lie on distinct limbs of $v$ in $S$.  Let $d_1=d_T(u,v_1)$, $d_2=d_T(u,v_2)$, $d_3=d_T(v,v_3)$, and $d_4=d_T(v,v_4)$. We may assume that $d_2 \le d_1$ and $d_3 \le d_4$. Let $w$ be the vertex distance $d_2+1$ from $u$ on the $u$--$v_{11}$ subpath of $T$, and let $z$ be the vertex distance $d_3+1$ from $v$ on the $v$--$v_{41}$ subpath of $T$.  Let $e_1=v_{21}w$, $e_2=v_{31}z$, and $F=T+\{e_1,e_2\}$. We show that $\{v_{11}, v_{41}\}$ resolves $F$.

Let $x,y$ be any two distinct vertices of $F$. Let $B_2$ be the branch of $T$ at $u$ containing $v_2$ and let $B_3$ be the branch of $T$ at $v$ containing $v_3$. Let $H_1=F-(V(B_2)-u)$, and let $H_2=F-(V(B_3)-v)$. Note that $H_1$ and $H_2$ are isometric subgraphs of $F$. From Subcase 2.1 of Theorem \ref{tree}, the set $\{v_{11},v_{41}\}$ resolves both $H_1$ and $H_2$.  So we may assume, without loss of generality, that $x$ is in $B_2-u$, and $y$ is in $B_3-v$.

Let $x'$ be the vertex on the $u$--$v_2$ path of $T$ that is closest to $x$, and let $y'$ be the vertex on the $v$--$v_3$ path of $T$ that is closest to $y$.  Let $k=d_F(u,x')$, $\ell=d_F(v,y')$, and $m=d_F(u,v)$.  Note that both $k$ and $\ell$ must be positive.  Then we have
\begin{align*}
d_F(v_{11},x) &= 2+d_1-k+d_F(x',x),\\
d_F(v_{41}, x)&= 1+d_4+m+k+d_F(x',x),\\
d_F(v_{11}, y)&= 1+d_1+m+\ell+d_F(y',y), \text{ and}\\
d_F(v_{41}, y)&= 2+d_4-\ell+d_F(y',y).
\end{align*}
Suppose that $x$ and $y$ are not resolved by the set $\{v_{11},v_{41}\}$.  Then we have $d_F(v_{11},x)=d_F(v_{11},y)$ and $d_F(v_{41},x)=d_F(v_{41},y)$.  Thus we have
\begin{align*}
d_F(v_{11},x)+d_F(v_{41},x)=d_F(v_{11},y)+d_F(v_{41},y),
\end{align*}
from which it follows that $d_F(y,y')=d_F(x,x')$.  Now since
\[
2+d_1-k+d_F(x',x)=d_F(v_{11},x)=d_F(v_{11},y)=1+d_1+m+\ell+d_F(y',y),
\]
we find that $k+\ell+m=1.$  However, since $k$ and $\ell$ must be positive, this is impossible.

\smallskip

\noindent \textbf{Subcase 4.3.2:} $S$ contains a single major vertex $v$.  This case is illustrated in Figure~\ref{Subcase4.3.2}.

\smallskip

\noindent
Then $v$ has degree $4$ in $S$.  For $i\in\{1,2,3,4\}$, let $d_i=d_T(v,v_i)$. We may assume that $d_2 \le d_1$ and $d_3 \le d_4$. Let $w$ be the vertex distance $d_2+1$ from $v$ on the $v$--$v_{11}$ subpath of $T$, and let $z$ be the vertex distance $d_3+1$ from $v$ on the $v$--$v_{41}$ subpath of $T$. Let $e_1=v_{21}w$, $e_2=v_{31}z$, and $F=T+\{e_1,e_2\}$. We claim that $\{v_{11}, v_{41}\}$ resolves $F$.  The proof is the same as the proof of Subcase 4.3.1, with $u$ set equal to $v$. (Note that the proof still works with $m=d_F(u,v)=0$.)
%\smallskip
%	
%In all cases, if $T$ is a tree with $\beta(T)=4$, then there exist two edges $\{e_1,e_2\}=E$ such that $\beta(T+E)=2$.
\end{proof}

We have shown that if $T$ is a tree with dimension $2$, $3$ or $4$, then $T$ has threshold dimension $2$.  This result does not extend to trees of dimension $5$, as the graph $K_{1,6}$ has metric dimension $5$ and threshold dimension $3$.
%If $T$ is the tree obtained from the path of order $5$ by attaching two leaves to each of its vertices, then $T$ has dimension $5$. Using a CI embedding of this tree in $P_6\boxtimes P_6$ similar to that of the tree shown in Figure \ref{tree4b} we see that $T$ has threshold dimension $2$.  This may not be immediately obvious, and by applying the methods of Theorem \ref{tree} this would not be attainable. However, by embedding the graph as shown in Figure \ref{tree4b}, we can easily see which edges should be added so that the resulting graph has dimension $2$. %(The centre vertex is solid to assist the reader in verifying the embedding.)
One way to see this is to use the following upper bound on the order $n$ of a graph with dimension $b$ and diameter $D$, proven by Hernando et al.~\cite{Hernandoetal2010}:
\begin{align} \label{bound}
	n \leq \left(\flooor{\frac{2D}{3}}+1\right)^b+b\sum_{i=1}^{\ceeil{D/3}}(2i-1)^{b-1}
\end{align}
In particular, if $G$ is a graph of dimension $2$ and diameter $2$, then $|V(G)|\leq 6$.  It follows that no graph in $\super(K_{1,6})$ has dimension $2$.  Are there other trees of dimension $5$ and threshold dimension greater than $2$?  While we show that every subdivision of $K_{1,6}$ has threshold dimension $2$, we are able to construct some other examples of trees with dimension $5$ and threshold dimension greater than $2$.

Let $T$ be a tree obtained by subdividing at least one edge of $K_{1,6}$.  Let $v$ be the vertex of degree $6$ in $T$, and let $L$ be a component of $T$ of order at least $2$.  Then for some sufficiently large path $P$, an embedding $\varphi$ of $T$ in $P\boxtimes P$ is shown in Figure~\ref{tree3}.  One can verify that $\varphi$ is in fact a $W$-resolved embedding of $T$ in $P\boxtimes P$, where the vertices of $W$ are coloured white.  Note that $L$ is the only component of $T$ that must have order at least $2$ in order for this property to hold.

\begin{figure} %subdivide tree3
	\centering	
	\begin{tikzpicture}[scale=0.6]
  \foreach \x in {0,1,3,4,5,6,7}
    \foreach \y in {0,1,3,4,5,6,7}
       {
       \vertex[opacity=0.25]  (\x\y) at (\x,\y) {};
       }
  \foreach \x in {0,3,4,5,6}
    \foreach \y in {0,3,4,5,6}
    {
    \pgfmathtruncatemacro{\a}{\x+1}
    \pgfmathtruncatemacro{\b}{\y+1}
    \path[opacity=0.25]
    (\x\y) edge (\a\b)
    (\x\b) edge (\a\y);
    }
\foreach \x in {0,3,4,5,6}
    \foreach \y in {0,1,4,5,6}
    {
    \pgfmathtruncatemacro{\a}{\x+1}
    \path[opacity=0.25]
    (\x\y) edge (\a\y)
    (\y\x) edge (\y\a);
    }
\foreach \x in {0,3,4,5,6}
{
\pgfmathtruncatemacro{\z}{\x+1}
\path[opacity=0.25]
(\x3) edge (\z3)
(3\x) edge (3\z)
(\x7) edge (\z7)
(7\x) edge (7\z)
;
}
\foreach \x in {0,1}
\foreach \y in {8,9}
{
\vertex[opacity=0.25]  (\x\y) at (\x,\y) {};
\vertex[opacity=0.25]  (\y\x) at (\y,\x) {};
}
\path[opacity=0.25]
(08) edge (09)
(18) edge (19)
(08) edge (18)
(09) edge (19)
(08) edge (19)
(18) edge (09)
(80) edge (90)
(81) edge (91)
(80) edge (81)
(90) edge (91)
(80) edge (91)
(81) edge (90);
\whitevertex (09) at (0,9) {};
\whitevertex (90) at (9,0) {};
\vertex (18) at (1,8) {};
\vertex (81) at (8,1) {};
\vertex (36) at (3,6) {};
\vertex (63) at (6,3) {};
\vertex (55) at (5,5) {};
\vertex (45) at (4,5) {};
\vertex (54) at (5,4) {};
\vertex (66) at (6,6) {};
\vertex (56) at (5,6) {};
\vertex (65) at (6,5) {};
\vertex (64) at (6,4) {};
\vertex (75) at (7,5) {};
\vertex (76) at (7,6) {};
\vertex (77) at (7,7) {};
\vertex (67) at (6,7) {};

\path[line width=1.5pt]
(09) edge (18)
(90) edge (81)
(36) edge (45)
(45) edge (55)
(63) edge (54)
(54) edge (55)
(55) edge (56)
(55) edge (66)
(55) edge (65)
(55) edge (64)
(56) edge (67)
(66) edge (77)
(65) edge (76)
(64) edge (75);

\node at (0.5,2) {\rotatebox[origin=c]{90}{$\cdots$}};
\node at (2,0.5) {$\cdots$};
\node at (0.5,7.5) {\rotatebox[origin=c]{90}{$\cdots$}};
\node at (7.5,0.5) {$\cdots$};
\node at (4,2) {\rotatebox[origin=c]{90}{$\cdots$}};
\node at (2,4) {$\cdots$};
\node at (4,8) {\rotatebox[origin=c]{90}{$\cdots$}};
\node at (8,4) {$\cdots$};
\node at (8,8) {\rotatebox[origin=c]{45}{$\cdots$}};
\node at (8,7) {\rotatebox[origin=c]{45}{$\cdots$}};
\node at (8,6) {\rotatebox[origin=c]{45}{$\cdots$}};
\node at (7,8) {\rotatebox[origin=c]{45}{$\cdots$}};
\node at (2,7) {\rotatebox[origin=c]{135}{$\cdots$}};
\node at (7,2) {\rotatebox[origin=c]{135}{$\cdots$}};
\draw[line width=1.5pt] (18) -- (1.5,7.5);
\draw[line width=1.5pt]  (81) -- (7.5,1.5);
\draw[line width=1.5pt]  (36) -- (2.5,6.5);
\draw[line width=1.5pt]  (63) -- (6.5,2.5);
\draw[line width=1.5pt]  (76) -- (7.5,6.5);
\draw[line width=1.5pt]  (77) -- (7.5,7.5);
\draw[line width=1.5pt]  (67) -- (6.5,7.5);
\draw[line width=1.5pt]  (75) -- (7.5,5.5);
\draw[dashed,rotate around={-45:(7,2)}] (7,2) ellipse (3.25 and 0.4);
\node at (8.5,2) {$L$};
\end{tikzpicture}
	\caption{A $W$-resolved embedding of any subdivision of $K_{1,6}$.}
	\label{tree3}
\end{figure}

Now consider the recursive sequence of trees $\{T_k\}_{k\geq 1}$ whose first four members are illustrated in Figure~\ref{5th3}.  It is straightforward to verify that for every $k\geq 1$, the tree $T_k$ has dimension $5$, diameter $2k$, and order $(5k+2)(k+1)/2$.  By an argument similar to the one used above for $K_{1,6}$, one can show that the trees $T_2$, $T_3$, and $T_4$ have threshold dimension greater than $2$.  However, this  is not the case for $T_5$. A $W$-resolved embedding of the tree $T_5$ in $P_{11}\boxtimes P_{11}$ is shown in Figure~\ref{T5}.  Hence, by Theorem \ref{Correspondence}, we have $\tau(T_5)=2$.

\begin{figure}[htb]
	\centering
		\begin{subfigure}[t]{0.2\textwidth}
		\centering
		\begin{tikzpicture}[scale=0.5, rotate=270]
		\pgfmathsetmacro{\n}{4}
		\pgfmathtruncatemacro{\m}{\n+1}
		\vertex (v) at (0:0) {};
		
		\foreach \x in {1, ..., \m}
		{
			\pgfmathtruncatemacro{\p}{60*\x}
			\pgfmathsetmacro{\y}{1}
			
			\vertex (\x\y) at (\p:\y) {};

		}
		\vertex (m) at (0:1) {};
		\path
		(v) edge (m);
		
		\foreach \x in {1, ..., \m}
		{
			\path
			(\x1) edge (v);
		}
		
		\end{tikzpicture}
		\caption{The tree $T_1$.}
		\label{5th3z}
	\end{subfigure}%
	\begin{subfigure}[t]{0.2\textwidth}
		\centering
		\begin{tikzpicture}[scale=0.5, rotate=270]
		\pgfmathsetmacro{\n}{4}
		\pgfmathtruncatemacro{\m}{\n+1}
		\vertex (v) at (0:0) {};
		
		\foreach \x in {1, ..., \m}
		{
			\pgfmathtruncatemacro{\p}{60*\x}
			\pgfmathsetmacro{\y}{1}
			
			\vertex (\x\y) at (\p:\y) {};

			\pgfmathsetmacro{\y}{2}
			\pgfmathtruncatemacro{\a}{\p-15}
			\pgfmathtruncatemacro{\b}{\p+15}
			
			\vertex(a\x) at (\a:\y) {};
			\vertex(b\x) at (\b:\y) {};
		}
		\vertex (m) at (0:1) {};
		\vertex (n) at (0:2) {};
		\path
		(v) edge (m)
		(m) edge (n);
		
		\foreach \x in {1, ..., \m}
		{
			\path
			(\x1) edge (v)
			(a\x) edge (\x1)
			(b\x) edge (\x1);
		}
		
		\end{tikzpicture}
		\caption{The tree $T_2$.}
		\label{5th3a}
	\end{subfigure}%
		\begin{subfigure}[t]{0.3\textwidth}
		\centering
		\begin{tikzpicture}[scale=0.5, rotate=270]

		\pgfmathsetmacro{\n}{4}
		\pgfmathtruncatemacro{\m}{\n+1}
		\vertex (v) at (0:0) {};
		
		\foreach \x in {1, ..., \m}
		{
			\pgfmathtruncatemacro{\p}{60*\x}
			\begin{scope}[shift=(\p:1)]
			\pgfmathsetmacro{\y}{0}
			
			\vertex (\x\y) at (\p:\y) {};

			\pgfmathsetmacro{\y}{1}
			\pgfmathtruncatemacro{\a}{\p-30}
			\pgfmathtruncatemacro{\b}{\p+30}
			
			\vertex(a\x) at (\a:\y) {};
			\vertex(b\x) at (\b:\y) {};
			
			\pgfmathsetmacro{\y}{2}
			\pgfmathtruncatemacro{\c}{\p-30}
			\pgfmathtruncatemacro{\d}{\p}
			\pgfmathtruncatemacro{\e}{\p+30}
			
			\vertex(c\x) at (\c:\y) {};
			\vertex(d\x) at (\d:\y) {};
			\vertex(e\x) at (\e:\y) {};
			
			\end{scope}
		}
		\vertex (m1) at (0:1) {};
		\vertex (m2) at (0:2) {};
		\vertex (m3) at (0:3) {};
		\path
		(v) edge (m1)
		(m) edge (m2)
		(m2) edge (m3);
		
		\foreach \x in {1, ..., \m}
		{
			\path
			(\x1) edge (v)
			(a\x) edge (\x1)
			(b\x) edge (\x1)
			(a\x) edge (c\x)
			(a\x) edge (d\x)
			(b\x) edge (e\x)
			;
			
		}
		
		\end{tikzpicture}
		\caption{The tree $T_3$.}
		\label{5th3b}
	\end{subfigure}%
	\begin{subfigure}[t]{0.3\textwidth}
		\centering
		\begin{tikzpicture}[scale=0.5, rotate=270]

		\pgfmathsetmacro{\n}{4}
		\pgfmathtruncatemacro{\m}{\n+1}
		\vertex (v) at (0:0) {};
		
		\foreach \x in {1, ..., \m}
		{
			\pgfmathtruncatemacro{\p}{60*\x}
			\begin{scope}[shift=(\p:1)]
			\pgfmathsetmacro{\y}{0}
			
			\vertex (\x\y) at (\p:\y) {};

			\pgfmathsetmacro{\y}{1}
			\pgfmathtruncatemacro{\a}{\p-30}
			\pgfmathtruncatemacro{\b}{\p+30}
			
			\vertex(a\x) at (\a:\y) {};
			\vertex(b\x) at (\b:\y) {};
			
			\pgfmathsetmacro{\y}{2}
			\pgfmathtruncatemacro{\c}{\p-30}
			\pgfmathtruncatemacro{\e}{\p+30}
			
			\vertex(c\x) at (\c:\y) {};
			\vertex(d\x) at (\p:\y) {};
			\vertex(e\x) at (\e:\y) {};
			
			\pgfmathsetmacro{\y}{3}
			\pgfmathtruncatemacro{\f}{\p+10}
			\pgfmathtruncatemacro{\g}{\p-10}
			
			\vertex(f\x) at (\c:\y) {};
			\vertex(l\x) at (\f:\y) {};
			\vertex(g\x) at (\g:\y) {};
			\vertex(h\x) at (\e:\y) {};
			
			\end{scope}
		}
		\vertex (m1) at (0:1) {};
		\vertex (m2) at (0:2) {};
		\vertex (m3) at (0:3) {};
		\vertex (m4) at (0:4) {};
		\path
		(v) edge (m1)
		(m) edge (m2)
		(m2) edge (m3)
		(m3) edge (m4);
		
		\foreach \x in {1, ..., \m}
		{
			\path
			(\x1) edge (v)
			(a\x) edge (\x1)
			(b\x) edge (\x1)
			(a\x) edge (c\x)
			(a\x) edge (d\x)
			(b\x) edge (e\x)
			(c\x) edge (f\x)
			(c\x) edge (g\x)
			(e\x) edge (h\x)
			(d\x) edge (l\x)
			;
			
		}
		
		\end{tikzpicture}
		\caption{The tree $T_4$.}
		\label{5th3d}
	\end{subfigure}
	\caption{The first four trees of the recursive sequence $\{T_k\}_{k\geq 1}$.}
	\label{5th3}
\end{figure}
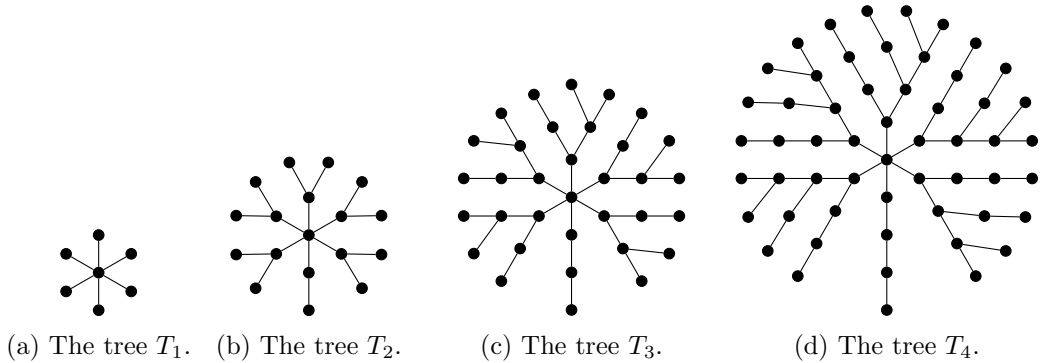

\begin{figure}[htb]
	\centering
	\begin{subfigure} [t]{0.4\textwidth}
		\centering
		\begin{tikzpicture}[scale=0.5, rotate=270]
		\pgfmathsetmacro{\n}{4}
		\pgfmathtruncatemacro{\m}{\n+1}
		\vertex (v) at (0:0) {};
		
		\foreach \x in {1, ..., \m}
		{
			\pgfmathtruncatemacro{\p}{60*\x}
			\begin{scope}[shift=(\p:1)]
			\pgfmathsetmacro{\y}{0}
			
			\vertex (\x\y) at (\p:\y) {};

			\pgfmathsetmacro{\y}{1}
			\pgfmathtruncatemacro{\a}{\p-30}
			\pgfmathtruncatemacro{\b}{\p+30}
			
			\vertex(a\x) at (\a:\y) {};
			\vertex(b\x) at (\b:\y) {};
			
			\pgfmathsetmacro{\y}{2}
			\pgfmathtruncatemacro{\c}{\p-30}
			\pgfmathtruncatemacro{\d}{\p+30}
			
			\vertex(c\x) at (\c:\y) {};
			\vertex(d\x) at (\p:\y) {};
			\vertex(e\x) at (\d:\y) {};
			
			\pgfmathsetmacro{\y}{3}
			\pgfmathtruncatemacro{\e}{\p-10}
			\pgfmathtruncatemacro{\f}{\p+10}
			
			\vertex(f\x) at (\e:\y) {};
			\vertex(r\x) at (\f:\y) {};
			\vertex(g\x) at (\c:\y) {};
			\vertex(h\x) at (\d:\y) {};
			
			\pgfmathsetmacro{\y}{4}
			\pgfmathtruncatemacro{\g}{\p-15}
			\pgfmathtruncatemacro{\h}{\p+15}
			
			\vertex(i\x) at (\g:\y) {};
			\vertex(j\x) at (\c:\y) {};
			\vertex(l\x) at (\p:\y) {};
			\vertex(k\x) at (\h:\y) {};
			\vertex(x\x) at (\d:\y) {};
			
			\end{scope}
		}
		\vertex (m1) at (0:1) {};
		\vertex (m2) at (0:2) {};
		\vertex (m3) at (0:3) {};
		\vertex (m4) at (0:4) {};
		\vertex (m5) at (0:5) {};
		\path
		(v) edge (m1)
		(m) edge (m2)
		(m2) edge (m3)
		(m3) edge (m4)
		(m4) edge (m5);
		
		\foreach \x in {1, ..., \m}
		{
			\path
			(\x1) edge (v)
			(a\x) edge (\x1)
			(b\x) edge (\x1)
			(a\x) edge (c\x)
			(a\x) edge (d\x)
			(b\x) edge (e\x)
			(c\x) edge (f\x)
			(c\x) edge (g\x)
			(e\x) edge (h\x)
			(d\x) edge (r\x)
			(i\x) edge (g\x)
			(j\x) edge (g\x)
			(k\x) edge (r\x)
			(l\x) edge (f\x)
			(h\x) edge (x\x)
			;
			
		}
		
		\end{tikzpicture}
		\caption{The tree $T_5$.}
		\label{5th3c}
	\end{subfigure}%
	\hspace{1cm}%
	\begin{subfigure}[t]{0.4\textwidth}
		\centering
\begin{tikzpicture}[scale=0.5]
  \foreach \x in {0,...,10}
    \foreach \y in {0,...,10}
       {
       \vertex[opacity=0.25]  (\x\y) at (\x,\y) {};
       }
  \foreach \x in {0,...,9}
    \foreach \y in {0,...,9}
    {
    \pgfmathtruncatemacro{\a}{\x+1}
    \pgfmathtruncatemacro{\b}{\y+1}
    \path[opacity=0.25]
    (\x\y) edge (\a\b)
    (\x\b) edge (\a\y);
    }
\foreach \x in {0,...,9}
    \foreach \y in {0,...,10}
    {
    \pgfmathtruncatemacro{\a}{\x+1}
    \path[opacity=0.25]
    (\x\y) edge (\a\y)
    (\y\x) edge (\y\a);
    }
\foreach \x in {0,...,10}
{
\node[left] at (0,\x) {\scriptsize \x};
\node[below] at (\x,0) {\scriptsize \x};
}
	
	\vertex (v) at (5,5) {};
	
	\vertex(1x1) at (0,7) {};
	\vertex(2x1) at (1,7) {};
	\vertex(21x1) at (1,8) {};
	\vertex(3x1) at (2,7) {};
	\vertex(31x1) at (2,8) {};
	\vertex(32x1) at (2,9) {};
	\vertex(4x1) at (3,7) {};
	\vertex(41x1) at (3,8) {};
	\vertex(42x1) at (3,9) {};
	\vertex(43x1) at (3,10) {};
	\vertex(5x1) at (4,6) {};
	\vertex(51x1) at (4,7) {};
	\vertex(52x1) at (4,8) {};
	\vertex(53x1) at (4,9) {};
	\vertex(54x1) at (4,10) {};
	
	\vertex(1x2) at (7,0) {};
	\vertex(2x2) at (7,1) {};
	\vertex(21x2) at (8,1) {};
	\vertex(3x2) at (7,2) {};
	\vertex(31x2) at (8,2) {};
	\vertex(32x2) at (9,2) {};
	\vertex(4x2) at (7,3) {};
	\vertex(41x2) at (8,3) {};
	\vertex(42x2) at (9,3) {};
	\vertex(43x2) at (10,3) {};
	\vertex(5x2) at (6,4) {};
	\vertex(51x2) at (7,4) {};
	\vertex(52x2) at (8,4) {};
	\vertex(53x2) at (9,4) {};
	\vertex(54x2) at (10,4) {};
	
	\vertex(1x3) at (1,6) {};
	\vertex(2x3) at (2,6) {};
	\vertex(21x3) at (2,5) {};
	\vertex(3x3) at (3,6) {};
	\vertex(31x3) at (3,5) {};
	\vertex(32x3) at (3,4) {};
	\vertex(4x3) at (4,5) {};
	\vertex(41x3) at (4,4) {};
	\vertex(42x3) at (5,3) {};
	\vertex(43x3) at (4,3) {};
	\vertex(5x3) at (5,4) {};
	\vertex(51x3) at (6,3) {};
	\vertex(52x3) at (6,2) {};
	\vertex(53x3) at (6,1) {};
	\vertex(54x3) at (5,2) {};
	
	\vertex(1x4) at (10,5) {};
	\vertex(2x4) at (9,5) {};
	\vertex(21x4) at (10,6) {};
	\vertex(3x4) at (8,5) {};
	\vertex(31x4) at (9,6) {};
	\vertex(32x4) at (10,7) {};
	\vertex(4x4) at (7,5) {};
	\vertex(41x4) at (8,6) {};
	\vertex(42x4) at (9,7) {};
	\vertex(43x4) at (10,8) {};
	\vertex(5x4) at (6,5) {};
	\vertex(51x4) at (7,6) {};
	\vertex(52x4) at (8,7) {};
	\vertex(53x4) at (9,8) {};
	\vertex(54x4) at (10,9) {};
	
	\vertex(1x5) at (5,10) {};
	\vertex(2x5) at (5,9) {};
	\vertex(21x5) at (6,10) {};
	\vertex(3x5) at (5,8) {};
	\vertex(31x5) at (6,9) {};
	\vertex(32x5) at (7,10) {};
	\vertex(4x5) at (5,7) {};
	\vertex(41x5) at (6,8) {};
	\vertex(42x5) at (7,9) {};
	\vertex(43x5) at (8,10) {};
	\vertex(5x5) at (5,6) {};
	\vertex(51x5) at (6,7) {};
	\vertex(52x5) at (7,8) {};
	\vertex(53x5) at (8,9) {};
	\vertex(54x5) at (9,10) {};
	
	\foreach \x in {6,7,8,9, 10}
	{
	\vertex (\x\x) at (\x,\x) {};
	}
	\path[line width=1.5pt]
	(v) edge (66)
	(66) edge (77)	
	(77) edge (88)
	(88) edge (99)
	(99) edge (1010);
	
	\foreach \x in {1, ..., 5}
	{
	\path[line width=1.5pt]
	(1x\x) edge (2x\x)
	(2x\x) edge (3x\x)
	(3x\x) edge (4x\x)
	(4x\x) edge (5x\x)
	(v) edge (5x\x)
	
	(2x\x) edge (21x\x)
	(3x\x) edge (31x\x)
	(31x\x) edge (32x\x)
	(4x\x) edge (41x\x)
	(41x\x) edge (42x\x)
	(42x\x) edge (43x\x)
	(5x\x) edge (51x\x)
	(51x\x) edge (52x\x)
	(52x\x) edge (53x\x)
	(53x\x) edge (54x\x)
	;
	}
	\whitevertex at (7,0) {};
	\whitevertex at (0,7) {};
	\end{tikzpicture}
	\caption{An embedding of $T_5$ in $P_{11}\boxtimes P_{11}$.}
	\end{subfigure}
	\caption{A $W$-resolved embedding of the tree $T_5$.  The vertices of $W$ are coloured white.}
	\label{T5}
\end{figure}
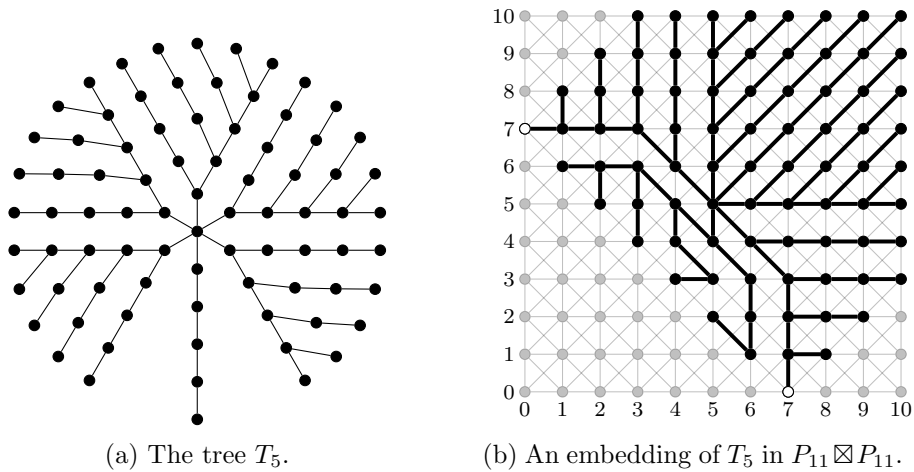

In fact, one can show that for all $k\geq 5$, the tree $T_k$ satisfies (\ref{bound}) for $b=2$.  However, it does not immediately follow that $\tau(T_k)=2$ for all $k\geq 5$.  We provide an example which illustrates that a tree $T$ of diameter $D$ may satisfy (\ref{bound}) for $b=2$, and still not have threshold dimension $2$.  Let $T$ be the tree obtained by subdividing every edge of the star $K_{1,8}$ exactly once. Then $T$ has diameter $4$ and order $17$, so (\ref{bound}) holds for $T$ in the case $b=2$. However, it can be verified that $T$ has threshold dimension at least $3$.

%These examples serve to highlight the difficulty in determining the threshold dimension of trees. Trees with very similar structures can have different threshold dimension, and different techniques can help in finding the threshold dimension of the various cases.

\section{Conclusion}

We have not considered the computational complexity of determining the threshold dimension of a graph.  For a graph $G$ and a positive integer $b$, consider the following decision problems:
\begin{enumerate}[label=(\alph*)]
\item Is there a set $B\subseteq V(G)$ such that $B$ is a basis for some graph $H\in\mathcal{U}(G)$?
\item Is there a graph $H\in\super(G)$ and a set $B\subseteq V(G)$ of cardinality $b$ such that $B$ resolves $H$?
\end{enumerate}
While it is easy to see that Problem (b) is in NP, it is not immediately clear whether or not Problem (a) is in NP.  We ask the following questions.

\begin{question}
Is Problem~(a) NP-hard, even for trees?  Is Problem~(b) NP-complete, even for trees?
\end{question}

In Section~\ref{trees}, we demonstrated that every tree with dimension $2$, $3$, or $4$ must have threshold dimension $2$.  While we gave some examples of trees of dimension $5$ and threshold dimension greater than $2$, no infinite family of trees with dimension $5$ and threshold dimension greater than $2$ is known.  We suspect that the following question has a positive answer.

\begin{question}
Are there only finitely many trees $T$ with $\beta(T)=5$ and $\tau(T)>2$?
\end{question}

\end{document}